\documentclass[11pt]{amsart}
\usepackage{amssymb}
\usepackage{amsmath}
\usepackage{amscd}
\usepackage{curves}
\usepackage{epsfig}
\usepackage{bbm}
\usepackage{geometry}
\usepackage{graphicx}
\usepackage{graphics}
\usepackage{pstricks}
\usepackage{pstricks-add}
\usepackage{pst-grad}
\usepackage{pst-plot}
\usepackage{verbatim}
\usepackage{latexsym}
\usepackage{eucal}
\usepackage{amsfonts}
\usepackage{enumerate}
\usepackage[final]{pdfpages}
\usepackage{epsfig}
\usepackage{epstopdf}
\usepackage{rotating}
\usepackage{cite}
\usepackage{color}
\usepackage{fancybox}
\usepackage{pstricks}
\usepackage{pst-plot}
\usepackage{pstricks-add}
\usepackage{epsf}

\numberwithin{equation}{section}
\newtheorem{theorem}{Theorem}[section]
\newtheorem{lemma}[theorem]{Lemma}
\newtheorem{prop}[theorem]{Proposition}

\newtheorem{definition}[theorem]{Definition}

\def \bpf {\begin{proof}}
\def \epf {\end{proof}}
\def \beq {\begin{equation*}}
\def \eeq {\end{equation*}}
\def \bsp{\begin{split}}
\def \esp{\end{split}}

\def \geqs { \gtrsim}

\def \CI {{C^\infty}}
\def \Coi {{C_0^\infty}}

\def \wt {\widetilde}

\def \mcb {{\mathcal B}}

\def \mce {{\mathcal E}}
\def \mcf {{\mathcal F}}

\def \mch {{\mathcal H}}

\def \mcl {{\mathcal L}}
\def \mcm {{\mathcal M}}
\def \mcw {{\mathcal W}}

\def \mcp {{\mathcal P}}
\def \mcq {{\mathcal Q}}
\def \mcr {{\mathcal R}}

\def \mcu {{\mathcal U}}
\def \mcv {{\mathcal V}}

\def \zed {{\mathcal Z}}
\def \mcy {{\mathcal Y}}

\def \mr {{\mathbb R}}
\def \mn {{{\mathbb N}_0}}

\def\ha {\frac{1}{2}}

\def\tha {\frac{3}{2}}
\def\fha {\frac{5}{2}}
\def\sha {\frac{7}{2}}

\def \msi {{m-\frac{n}{4}+\frac{1}{2}}}

\def \wtW {\widetilde{W}}

\def \ka {\kappa}

\def \Id {\operatorname{Id}}

\def \loc {\operatorname{loc}}

\def \diag{\operatorname{Diag}}

\def \mrn {{\mathbb R}^n}

\def \idm {{\stackrel{o}{\operatorname{I}}}{}^{m-\frac{n}{4}+\ha}}

\def \ido {{\stackrel{o}{\operatorname{I}}}}

\def \ga {{\gamma}}
\def \eps {\varepsilon}   
\def \vphi {\varphi}  
\def \vkap {\varkappa}
   
\def \La {\Lambda}   
\def \lan {\langle}   
\def \ran {\rangle}   
\def \del {\delta}   

\def \p {\partial}

\def \geqs {\gtrsim}

\def \novt {\frac{n}{2}}
\def \novf {\frac{n}{4}}

\def \beqq {\begin{equation}}
\def \eeqq {\end{equation}}

\def \mck {\mathcal{K}}

\numberwithin{equation}{section}

\begin{document}
\title[Interactions of Semilinear Waves]{ Interactions of Semilinear Progressing Waves in Two or More Space Dimensions}
\author{ Ant\^onio S\'a Barreto }
\address{Ant\^onio S\'a Barreto\newline
\indent Department of Mathematics, Purdue University \newline
\indent 150 North University Street, West Lafayette Indiana,  47907, USA}
\email{sabarre@purdue.edu}
\begin{abstract}   We show that singularities form after the interaction of three  transversal semilinear conormal waves.  Our results hold for  space dimensions two and higher, and for arbitrary $\CI$ nonlinearity. The case of two space dimensions in which the nonlinearity  is a polynomial  was studied by the author and Yiran Wang.  
\end{abstract}

\maketitle
\tableofcontents
\section{Introduction}

When singularities of nonlinear waves interact, they will produce additional singularities in a way which in general is very hard to predict.  When the singularities are conormal, the behavior of the newly formed singularities is more tractable. We study this phenomenon when three transversal conormal waves interact  in dimensions greater than or equal to three, and we show that this will produce new singularities on the hypersurface emanating from the submanifold where the three waves interact, unless some degeneracy occurs.  We also show that the new singularities do contain information about the nonlinear term, raising the possibility this will find applications in the study of inverse problems.

We consider solutions to  $P(y,D) u=\mcy(y) f(y,u),$  $y \in \mr^n,$ $n\geq 3,$ where $P(y,D)$ is a second order strictly hyperbolic operator, $\mcy\in C_0^\infty,$  and $f\in C^\infty.$ We assume  that for negative times $\mcy=0$ and that the solution $u$  is  the superposition of three elliptic conormal waves intersecting transversally at a codimension three submanifold $\Gamma,$ which will intersect the support of the nonlinearity. Bony \cite{Bony3} showed that as long as the incoming waves do not have caustics, no new singularities are formed before the triple interaction; transversal interactions of two waves do not produce new singularities.
Melrose and Ritter \cite{MelRit}, and Bony \cite{Bony5,Bony6} have shown that after a triple interaction occurs on the support of the nonlinearity, the solution $u$ may have additional singularities on the characteristic hypersurface $\mcq$ emanating from $\Gamma,$ and $u$ will be conormal to $\mcq$ away from $\Gamma$ and the incoming hypersurfaces, see Fig.\ref{Fig1}.   The papers \cite{Bony5,Bony6,MelRit} are in fact about the three dimensional case, however, their methods apply in higher dimensions, as pointed out in \cite{Mel-EP}.  However, unlike the three dimensional case, the submanifold $\Gamma$ is not necessarily contained on a level curve of the time function, so the more appropriate version of the result of Melrose and Ritter and Bony to be applied in the  higher dimensional case is that if the solution $u$ to $Pu=f(y,u)$ is conormal (in a suitable sense) to $\Sigma_1\cup\Sigma_2\cup \Sigma_3 \cup \mcq,$ for $t<0,$ then $u$ is also conormal to 
$\Sigma_1\cup\Sigma_2\cup \Sigma_3 \cup \mcq$ for $t>0.$  This result is somewhat contained in \cite{Bony5,Bony6,MelRit}, but it is explicitly stated in \cite{SaB} including the case $Pu=f(y,u,Du).$

However,  these results  do not guarantee that singularities on   $\mcq$ will in fact exist.  Examples of formation of singularities after the triple interaction were given by Rauch and Reed \cite{RauRee} and Beals \cite{Beals4}.  S\'a Barreto and Wang \cite{SaWang1} proved a particular case of Theorem \ref{triple-0} below, when $n=3$ and the nonlinear term $f(y,u)$ is a polynomial in $u.$ The purpose of this paper is to extend the results of \cite{SaWang1} to arbitrary $\CI$ nonlinear terms $f(y,u)$ and higher dimensions.  We show that if  the initial data is elliptic of order $m,$ and if $(\p_u^3 f)(q,u(q))\not=0$ for some $q\in \Gamma,$  then  there is a neighborhood $U_q\subset \Gamma $ of $q$ such that $u$ will be elliptic roughly of order $3m$ along any null bicharacteristics on the conormal bundle to $\mcq$ which passes  $N^* U_q\setminus 0,$ as  long as they do not intersect the hypersurfaces and $\mcq$ is $\CI$ in a neighborhood of the ray.  In this paper we extend these results to general semilinear equations and dimensions greater than or equal to three.

Our results show that one can recover $(\p_u^3 f)(y, u(y)),$ for $y\in \Gamma,$ which can be viewed as an inverse result.   Kurylev, Lassas and Uhlmann \cite{KLU} were the first to use the propagation of singularities for semilinear equations to study inverse problems.  Several other papers have followed, see for example \cite{LUW,UW,Paternain} and references cited there. 

\section{The Framework of the Problem and Examples}

Let $\Omega\subset \mrn,$ $n\geq 3,$  be an open subset,  let $P(y,D)$ be a second order strictly hyperbolic operator and assume that  $\Omega\subset \mr^n$ is bicharacteristically convex with respect to $P(y,D).$    Let $t$ be a time function for $P(y,D)$ in $\Omega.$  This means that  there exists an open set $U$ such that $\Omega\subset U \subset \mr \times \mr^{n-1}$ such that
\begin{gather}
P(t,x, D)= \alpha(t,x)^2 \p_t^2-\sum_{jk} h_{jk}(t,x) \p_{x_j}\p_{x_k}, \;\ \alpha>0, \text{ in } U, \label{pdt}
\end{gather}
and $h_{jk}(t,x)\xi_j\xi_k$ is positive definite.

Let $\Sigma_j,$ $j=1,2,3,$ be $C^\infty$ hypersurfaces which are closed and characteristic for $P(y,D).$  Moreover, we assume that
the normals of the surfaces $\Sigma_j,$ $j=1,2,3,$ are linearly independent over the submanifolds
 \begin{gather}
  \begin{gathered}
 \Gamma_{jk} =\Gamma_j \cap \Gamma_k, \; j\not=k, \text{ and }
 \Gamma= \Sigma_1\cap \Sigma_2\cap \Sigma_3.
 \end{gathered} \label{intersections}
\end{gather}

  Let $v(y)= v_1(y)+ v_2(y)+ v_3(y),$ where $v_j(y)$ is a conormal distribution of appropriate order with respect to $\Sigma_j,$ $j=1,2,3,$ and assume that  
 
\begin{equation}
\begin{split}
P(y,D) & v_j(y) = 0, \ y \in \Omega, \;\ j=1,2,3.
\end{split} \label{Weq-0}
\end{equation}

We  will analyze the propagation of singularities of solutions $u(y) \in H^s_{\loc}(\Omega),$ $s>\frac{n}{2},$  of  semilinear wave equations of the form
\begin{equation}
\begin{split}
P(y,D) & u(y) =  \mcy(y) f(y,u(y)), \ y \in \Omega, \\
& u(y)= v(y),  \text{ for } t<-1,
\end{split} \label{Weq}
\end{equation}
where $t$ be a time function of $P(y,D)$ in $\Omega,$  $f\in C^\infty(\Omega\times \mr),$ $\mcy\in C_0^\infty(\Omega),$ $\mcy=0$ when $t<-1.$

Our results also apply, with minor changes of the proof, to  the forcing problem
\begin{equation}
\begin{split}
P(y,D) & u(y) = \ f(y,u(y))+g(y) , \ y \in \Omega, \\
& g(y)=u(y)= 0,  \text{ and } f(y,\bullet)=0 \text{ in } t<-1,
\end{split} \label{Weq-F}
\end{equation}
where $g(y)=g_1(y)+g_2(y)+g_3(y),$ $g_j$ conormal to $\Sigma_j,$ $j=1,2,3.$   This is the form of the equation used in the applications to the nonlinear inverse problems as in  \cite{Paternain,KLU,LUW,UW}.

Let $\Gamma$ be defined in \eqref{intersections} and let $N^*\Gamma\setminus 0$ denote its conormal bundle minus its zero section, and for each $q\in \Gamma,$  let  $N_q^*\Gamma\setminus 0$ denote  its fiber over $q.$  Let $p(y,\eta)$ denote the principal symbol of $P(y,D)$ and  let $H_p$  be its Hamilton vector field. For each $(q,\eta) \in (N^*_q\Gamma\setminus 0)\cap p^{-1}(0),$ let   $\gamma_{(q,\eta)}^+$ denote the forward null bicharacteristic for $p$ passing through 
$(q,\eta)$ and let
\begin{gather}
\La_q=\bigcup_{s>0}  \exp(s H_p) \left( (N_q^*\Gamma\setminus 0) \cap p^{-1}(0)\right)=  \bigcup_{(q,\eta)\in( N_q^*\Gamma\setminus 0) \cap p^{-1}(0)} \gamma_{(q,\eta)}^+, \label{laq}
\end{gather}
denote the flow-out of $(N_q^*\Gamma\setminus 0) \cap p^{-1}(0)$ by $H_p.$  Since $\Gamma$ has dimension $n-3,$ the dimension of the fiber of $N_q^*\Gamma\setminus 0$ is equal to three, and it is well known, see for example \cite{HormanderV3} that 
$\La_q$ is a three dimensional $C^\infty$ submanifold of $T^* \Omega\setminus 0$ which is isotropic, that is the canonical symplectic form vanishes on its tangent space. The manifold
\begin{gather}
\La= \bigcup_{q\in \Gamma}  \La_q, \label{defla}
\end{gather}
is a $C^\infty$ conic Lagrangian submanifold of $T^*\Omega \setminus 0.$   When $n=3,$ $\Gamma=\{q\}$ and $\La_q=\La.$

We will need to analyze the projection of the bicharacteristics $\gamma_{(q,\eta)}$ and the Lagrangian  $\La_q$ from $T^*\Omega\setminus 0$ to $\Omega.$  As usual,  we let 
\begin{gather*}
\Pi: T^* \Omega\setminus 0 \longrightarrow \Omega
\end{gather*}
denote the canonical projection, and for a bicharacteristic $\gamma_{(q,\eta)}^+,$ let 
\begin{gather}
\begin{gathered}
\sigma_{(q,\eta)}= \Pi(\gamma_{(q,\eta)}^+), \; (q,\eta)\in   (N_q^*\Gamma\setminus 0) \cap p^{-1}(0), \\
\mcq_q= \Pi(\La_q)= \bigcup_{\{\eta: \ (q,\eta) \in   (N_q^*\Gamma\setminus 0) \cap p^{-1}(0)\}} \sigma_{(q,\eta)}, \\
\mcq = \bigcup_{q \in \Gamma} \mcq_q.
\end{gathered}  \label{mcqq} 
\end{gather}  
In dimension $n=3,$ $\Gamma$ is a point and $\mcq_q=\mcq$ is the forward light cone for the operator $P$ over $\Gamma,$ and 
$\sigma_{(q,\eta)}$ is a characteristic line on the cone. It's well known that $\mcq\setminus \Gamma$ is a $C^\infty$ manifold in a neighborhood of $\Gamma.$

There exists one very important difference between the three and the higher dimensional cases. In the three dimensional case, the three waves intersect at a point and $\mcq$ is the light cone with vertex at this point, this is the only point where the cone $\mcq$ interacts with the three incoming waves.  In the general case, this is not true, see examples below.   Once the three incoming waves intersect  within the support of $\mcy(y) f(y,u),$ singularities will form on $\mcq$ and they will interact with the three incoming waves.  However, since $\Sigma$ is characteristic for $P(y,D),$ which is given by \eqref{pdt}, if locally $\Sigma=\{\vphi=0\}$ with $|(\p_t \vphi,\p_x\vphi)|>0,$ then 
\begin{gather*}
\alpha(t,x) (\p_t \vphi)^2-\sum_{jk} h_{jk}(t,x) \p_{x_j}\vphi \p_{x_k}\vphi=0,
\end{gather*}
and so, $|\p_t\vphi|>0.$  So we conclude that  the vector field $\p_t$ is transversal to $\Gamma.$  This gives a time orientation as time increases across $\Gamma.$  

The singularities produced by the interaction of three waves in $\mr^3,$ two space dimensions, is shown in Fig.\ref{Fig1}.  The figure shows the configuration for fixed time before and after the triple interaction.  The formation of singularities in $\mrn,$ $n>3,$ is much richer.  We give three  examples of the interaction of three plane waves in $\mr^4,$ which  illustrate   how $\Gamma$ and $\mcq$ can look like.

\begin{figure}
\centering
\psscalebox{.6} 
{
\begin{pspicture}(0,-2.9017625)(20.015244,2.9017625)
\psline[linecolor=black, linewidth=0.04](4.415244,2.8888159)(0.015243912,-2.2963693)
\psline[linecolor=black, linewidth=0.04](1.6152439,2.3036306)(1.6152439,2.3036306)(5.215244,-2.0963693)
\psline[linecolor=black, linewidth=0.04](11.215244,1.9036307)(8.015244,-2.0963693)
\psline[linecolor=black, linewidth=0.04](7.015244,1.5036306)(11.415244,-2.0963693)(11.415244,-2.0963693)
\psline[linecolor=black, linewidth=0.04](6.8152437,-0.49636933)(13.815244,-0.2963693)
\rput[bl](9.215244,-1.4963694){$t=0$}
\rput[bl](2.015244,-1.4963694){$t<-1$}
\rput[bl](15.923577,-1.4297026){$t>0$}
\psline[linecolor=black, linewidth=0.04](0.2152439,-0.49636933)(5.615244,-0.49636933)
\pscircle[linecolor=black, linewidth=0.04, dimen=outer](17.215244,-0.096369326){0.8}
\psline[linecolor=black, linewidth=0.04](14.215244,1.3036307)(18.615244,-2.2963693)(18.615244,-2.2963693)
\psline[linecolor=black, linewidth=0.04](14.015244,0.7036307)(20.015244,0.7036307)
\psline[linecolor=black, linewidth=0.04](18.615244,2.1036308)(17.215244,-2.8963692)
\rput[bl](14.015244,2.3036306){possible new singularities}
\rput[bl](14.015244,1.7036307){after the triple interaction}
\end{pspicture}
}
\caption{The interaction of three conormal plane waves in two space dimensions. The only possible singularities created by the triple interaction appear on the surface of the light cone.  Fig.\ref{Fig3WN}, Fig.\ref{wave1}  and Fig.\ref{figure4} below illustrate the higher dimensional cases.}
\label{Fig1}
\end{figure}
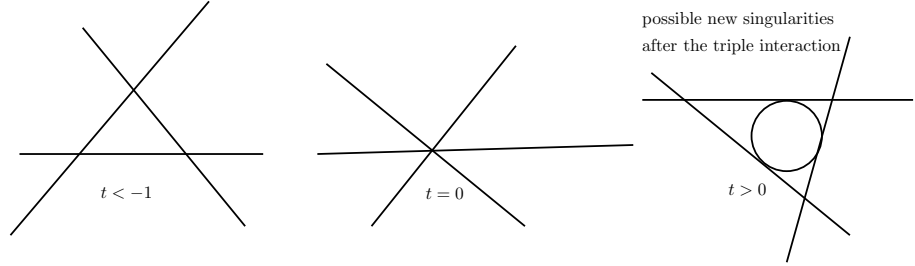

The operator is  the Minkowski wave operator in $\mr^4$ with coordinates $(t,x),$ $x=(x_1,x_2,x_3):$ 
\begin{gather*}
P(t,x)= \p_t^2-\p_{x_1}^2-\p_{x_2}^2-\p_{x_3}^2.
\end{gather*}

In the first example we take three plane waves:
\begin{gather}
\Sigma_1=\{t=x_1\}, \;\ \Sigma_2=\{t=x_2\} \text{ and } \Sigma_3=\{t=\frac{1}{\sqrt{2}}(x_1+x_2)\}, \label{Surf-1}
\end{gather}
which will intersect transversally at
\begin{gather*}
\Gamma=\{ t=0, x_1=0, x_2=0\}.
\end{gather*}
The conormal bundle to $\Gamma$ is
\begin{gather*}
N^*\Gamma\setminus 0=\{ t=0, x_1=0, x_2=0, \xi_3=0\},
\end{gather*}
and the flow out of $N^*\Gamma\setminus 0$ is given by
\begin{gather*}
\La= \{(t,x,\tau,\xi): \tau=\tau_0,\;  \xi=\xi_0, \; t=2\tau s, \; x_1=-2\xi_1 s, \; x_2=-2\xi_2 s,\;  x_3= x_{03}, \; \tau=|\xi|, \; \xi_3=0\}.
\end{gather*}
The projection of $\La$ to $\mr^4$ is given by
\begin{gather*}
\mcq=\{ (t,x) \in \mr^4: t=(x_1^2+x_2^2)^\ha\}.
\end{gather*}
This can be viewed as a fiber bundle over $\Gamma$ where the fiber over a point $(0,0,0,x_{03})\in \Gamma$ is the circle  $t=(x_1^2+x_2^2)^\ha.$  An observer sitting at $(0,0,x_3)$ will see a circular wave expanding with speed one, see Fig.\ref{Fig3WN}.  
\begin{figure}[h!]
\scalebox{.6} 
{
\begin{pspicture}(0,-3.0352242)(12.5,3.0094604)
\usefont{T1}{ptm}{m}{n}
\rput(0.6514551,-2.0655394){$x_1$}
\usefont{T1}{ptm}{m}{n}
\rput(3.6514552,2.7344606){$x_2$}
\usefont{T1}{ptm}{m}{n}
\rput{-15.98614}(0.8626451,1.7830542){\rput(6.761455,-2.1655395){$x_1^2+x_2^2=t^2$}}
\usefont{T1}{ptm}{m}{n}
\rput{-15.598908}(0.4574698,1.8194718){\rput(6.851455,-0.7455394){$\Gamma=\{x_1=x_2=0\}$}}
\usefont{T1}{ptm}{m}{n}
\rput{-20.62775}(1.7116534,4.045296){\rput(11.951455,-2.6655395){$x_3$}}
\psline[linewidth=0.04cm,linestyle=dashed,dash=0.16cm 0.16cm](1.38,1.3094605)(10.4,-1.0105394)
\psline[linewidth=0.04cm,linestyle=dashed,dash=0.16cm 0.16cm](1.4,-0.29053944)(10.26,-2.6305394)
\psline[linewidth=0.04cm](3.06,2.9894605)(3.0,0.029460559)
\psline[linewidth=0.04cm](3.0,0.029460559)(12.48,-2.4305394)
\psline[linewidth=0.04cm](2.98,0.04946056)(0.0,-1.8505394)
\psellipse[linewidth=0.04,linestyle=dashed,dash=0.16cm 0.16cm,dimen=outer](1.34,0.52946055)(0.5,0.82)
\psellipse[linewidth=0.04,linestyle=dashed,dash=0.16cm 0.16cm,dimen=outer](4.78,-0.33053944)(0.5,0.82)
\psellipse[linewidth=0.04,linestyle=dashed,dash=0.16cm 0.16cm,dimen=outer](10.32,-1.8505394)(0.5,0.82)
\end{pspicture} 
}
\caption{ The dotted line represents an expanding cylindrical wave, generated by the interaction of three plane waves  given by \eqref{Surf-1} in $\mr^4,$  viewed by an observer in $\mr^3$  as time increases.  The speed in which the radius of the wave expands is equal to one.}
\label{Fig3WN}
\end{figure}
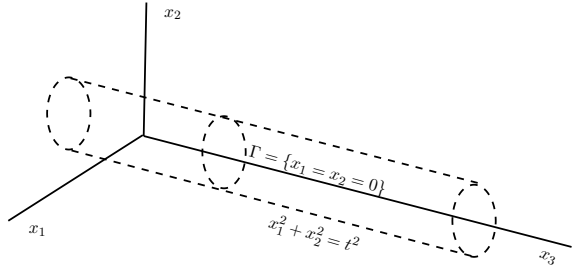

In the second example, we pick three plane waves,
\begin{gather}
\Sigma_j=\{t=x_j\}, \;\ j=1,2,3. \label{sl-sig}
\end{gather}
but in this case they meet at $\Gamma=\{t=x_1=x_2=x_3\},$ which is not contained at a level surface of $t.$ The conormal bundle to $\Gamma$ is given by
\begin{gather*}
N^*\Gamma=\{(t,x,\tau,\xi): t=x_1=x_2=x_3, \tau=\tau_0, \xi=\xi_0, \tau+\xi_1+\xi_2+\xi_3=0\},
\end{gather*}
and so the Lagrangian $\La$ is given by
\begin{gather*}
\La=\{ t=a+ 2\tau s, x_j= a-2\xi_j s, \; \xi_j=\xi_{j0}, \; \tau=\tau_0, \;  \tau+\xi_1+\xi_2+\xi_3=0, \; \tau=|\xi|, \ a, s\in \mr\}. 
\end{gather*}
Its projection to $\mr^4$ is given by
\begin{gather*}
\mcq=\{(3t-x_1-x_2-x_3)^2=(x_1-x_2-x_3+t)^2+(x_2-x_1-x_3+t)^2+(x_3-x_2-x_1+t)^2 \}.
\end{gather*}
One can check that $t-x_1-x_2-x_3=-2a$ on $\mcq,$ where the parameter $a$ gives the position of a point  on $\Gamma.$ So to consider the behavior of $\mcq$ for a fixed time $t,$ one should restrict the variable $a\in [t, A_1].$   The forward part of $\mcq$  for fixed $t$ and viewed by an observer in $\mr^3$ is part of a  cone with axis of symmetry $L=\{x_1=x_2=x_3\}$ and vertex at $(t,t,t),$  bounded by the planes $3t \leq x_1+x_2+x_3\leq t+2A_1,$  see figure Fig.\ref{figure4}.
\begin{figure}
\scalebox{.6} 
{
\begin{pspicture}(0,-5.3376164)(7.8818946,5.3376164)
\psline[linewidth=0.04cm](0.02,5.2223835)(0.0,-2.8776162)
\psline[linewidth=0.04cm](0.0,-2.9176164)(5.96,-5.3176165)
\psline[linewidth=0.04cm](0.04,-2.9176164)(6.72,0.0023837236)
\psline[linewidth=0.04cm](0.04,-2.8976164)(5.48,5.2023835)
\psline[linewidth=0.04cm,linestyle=dashed,dash=0.16cm 0.16cm](2.4,4.6423836)(2.0,0.08238372)
\psline[linewidth=0.04cm,linestyle=dashed,dash=0.16cm 0.16cm](2.06,0.062383723)(6.54,2.5823836)
\psbezier[linewidth=0.04,linestyle=dashed,dash=0.16cm 0.16cm](2.38,4.5223837)(2.38,3.7223837)(6.1190577,1.7152963)(6.54,2.6223838)(6.9609423,3.5294712)(3.0388865,5.3176165)(2.34,4.6023836)
\usefont{T1}{ptm}{m}{n}
\rput(2.591455,-0.092616275){$(t,t,t)$}
\usefont{T1}{ptm}{m}{n}
\rput(5.991455,-4.9126163){$x_1$}
\usefont{T1}{ptm}{m}{n}
\rput(0.57145506,4.967384){$x_3$}
\usefont{T1}{ptm}{m}{n}
\rput(6.831455,-0.47261629){$x_2$}
\usefont{T1}{ptm}{m}{n}
\rput(6.651455,4.887384){$x_1=x_1=x_3$}
\end{pspicture} 
}
\caption{Singularities produced by the intersection of three plane waves \eqref{sl-sig}. An observer in $\mr^3$ sees a conic shaped wave.}
\label{figure4}
\end{figure}
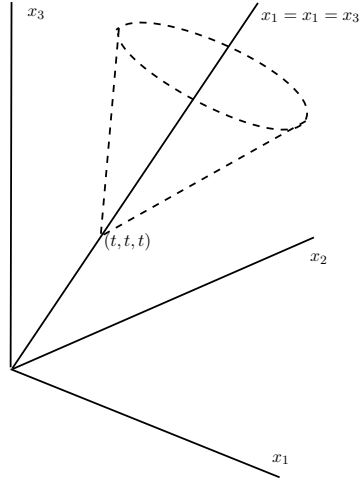

In the third example we pick  three spherical waves centered at at $p_1=(0,0,0,0),$ $p_2=(2a,0,0,0),$ and $p_3=(0,2b,0,0)$respectively.   These are represented by forward three light  cones with vertices at $p_j,$ $j=1,2,3:$
 \begin{gather*}
 t= (x_1^2+x_2^2+ x_3^2)^\ha,  \\
 t=((x_1-2a)^2+x_2^2+x_3^2)^\ha,\\
 t=(x_1^2+(x_2-2b)^2+x_3^2)^\ha.
 \end{gather*}
 These waves  will intersect transversally at the hyperbola
 \begin{gather*}
 \Gamma=\Gamma_{a,b}= \{ (t,x): x_1=a, x_2=b,  \;\  t=(x_3^2+ a^2+b^2)^\ha\},
  \end{gather*}
 whose conormal bundle is given by
 \begin{gather*}
 N^*\Gamma=\{(t,x,\tau,\xi): x_1=a, \; x_2=b, \; t=(x_3^2+a^2+b^2)^\ha, \; x_3\tau + t \xi_3=0\}.
 \end{gather*}
 The Lagrangian submanifold $\La$ obtained by the forward flow-out of $(N^*\Gamma\setminus 0) \cap \{p=0\},$ is given by
 \begin{gather*}
 \text{ for } t_0= |x_0|, \;  x_0=(a, b, x_{03}),  \;\ \tau_0=|\xi_0|, \;\ \xi_0=(\xi_{01}, \xi_{02}, \xi_{03}),\\
 \xi_1=\xi_{01}, \; \xi_2=\xi_{02}, \;  \xi_3=\xi_{03}, \;\ \tau=\tau_0, \;\ x_{03}\tau_0+ t_0 \xi_{03}=0, \\  
 x_1=a -2\xi_{01} s, \;\ x_2=b -2\xi_{02} s, \;\ x_3= x_{03}-2\xi_{03} s, \;\ t=t_0+2\tau_0 s, \;\ s \in \mr.
 \end{gather*}

 The projection of $\La_{Q}$ to $\mr^4$ is denoted by $\mcq$ and it is given by 
 \begin{gather}
\mcq=\mcq_{a,b}=\{(t, x): (x_1-a)^2+(x_2-b)^2= \left[(t^2-x_3^3)^\ha-(a^2+b^2)^\ha\right]^2, \;\ x_3=-\frac{x_{03}}{t_0} t\},\label{eq-mcq}
\end{gather}
which is again a fibered bundle over $\Gamma$ whose fibers over the point $(t_0,x_0)=(t_0,a,b,x_{03})\in \Gamma$ given by these equations.
\begin{figure}
\scalebox{.6} 
{
\begin{pspicture}(0,-7.578848)(16.68291,7.5588474)
\psline[linewidth=0.04cm](3.6210155,7.5388474)(3.5410156,-0.22115235)
\psline[linewidth=0.04cm](3.5410156,-0.24115235)(0.40101564,-4.1211524)
\psline[linewidth=0.04cm](3.6010156,-0.22115235)(13.081016,-0.22115235)
\psdots[dotsize=0.12](7.7610154,-1.6011523)
\psbezier[linewidth=0.04,linestyle=dashed,dash=0.16cm 0.16cm](7.841016,4.4588475)(6.1610155,3.6988478)(5.1410155,-0.20115234)(5.321016,-1.6011523)
\psbezier[linewidth=0.04,linestyle=dashed,dash=0.16cm 0.16cm](7.7410154,-7.0011525)(6.6010156,-6.401152)(5.4881234,-4.9989667)(5.2810154,-1.5811523)
\usefont{T1}{ptm}{m}{n}
\rput(0.4124707,-4.5161524){$x_1$}
\usefont{T1}{ptm}{m}{n}
\rput(13.01247,-0.056152344){$x_2$}
\usefont{T1}{ptm}{m}{n}
\rput(4.352471,7.223848){$x_3$}
\psline[linewidth=0.04cm,linestyle=dashed,dash=0.16cm 0.16cm](7.821016,-1.5811523)(10.581016,-1.5411524)
\usefont{T1}{ptm}{m}{n}
\rput(13.872471,-1.5961523){$x_3=0,$ radius $R=t-(a^2+b^2)^\ha$}
\usefont{T1}{ptm}{m}{n}
\rput(11.31247,4.6238475){$x_3=(t^2-a^2-b^2)^\ha,$ radius $R=0.$}
\psbezier[linewidth=0.04,linestyle=dashed,dash=0.16cm 0.16cm](7.841016,4.4388475)(9.281015,3.7788477)(10.389645,1.4163563)(10.6610155,-1.6611524)
\psbezier[linewidth=0.04,linestyle=dashed,dash=0.16cm 0.16cm](7.801016,-7.0411525)(9.001016,-6.2611523)(10.341016,-5.5611525)(10.641016,-1.5811523)
\usefont{T1}{ptm}{m}{n}
\rput(8.382471,-1.2961524){$(a,b)$}
\psellipse[linewidth=0.04,linestyle=dashed,dash=0.16cm 0.16cm,dimen=outer](7.881016,1.4188477)(2.16,0.24)
\psellipse[linewidth=0.04,linestyle=dashed,dash=0.16cm 0.16cm,dimen=outer](7.9210157,-3.9811523)(2.32,0.22)
\psellipse[linewidth=0.04,linestyle=dashed,dash=0.16cm 0.16cm,dimen=outer](7.9210157,-1.5711523)(2.74,0.43)
\psline[linewidth=0.04cm](7.821016,4.4588475)(7.801016,-7.0411525)
\usefont{T1}{ptm}{m}{n}
\rput(11.092471,-7.3561525){$x_3=-(t^2-a^2-b^2)^\ha,$ radius $R=0.$}
\end{pspicture} 
}
\caption{The dotted line shows the surface \eqref{eq-mcq} as $(x_1,x_2,x_3)$ vary for $t$ fixed.  Unlike the wave formed by the interaction of three plane waves considered above, which is an infinite cylinder,  three spherical waves intersect along a bounded curve for fixed time. The level sets of this surface for $\{x_3=c\}$ are circles centered on the line $\{x_1=a, x_2=b\}.$ }
\label{wave1}
\end{figure}
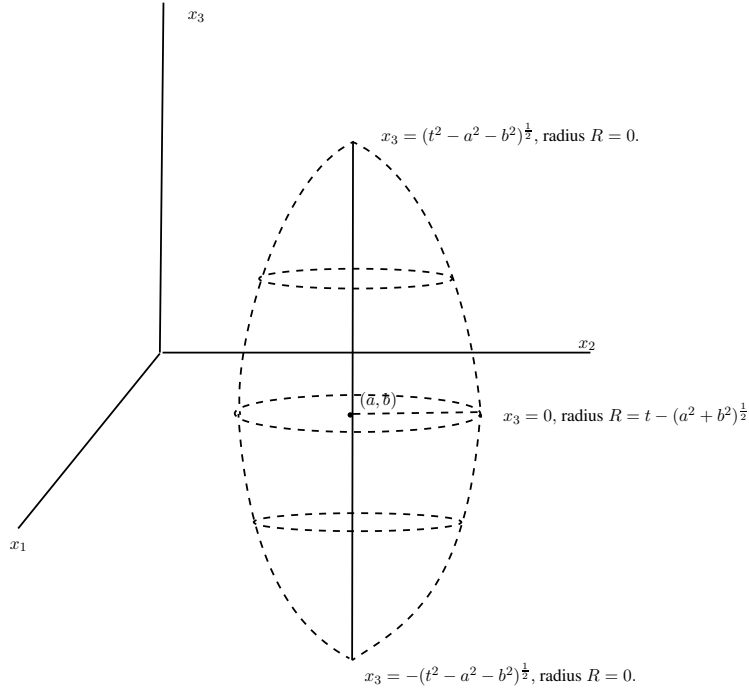

One can also think about $\mcq_{a,b}$  from the point of view of of an observer in $\mr^3.$  One has three spherical waves, centered at  $p_1,p_2$ and $p_3$ and expanding with speed one. They  first meet at the point $(a,b,0)$ which is equidistant from their centers $p_1,$ $p_2$ and $p_3,$ at a time $t_0=(a^2+b^2)^\ha,$ which is equal to the distance between any of the centers to the point of interaction.  After that, a wave centered at the line $L=\{x_1=a, x_2=b\}$ will form for times $t\geq (a^2+b^2)^\ha.$  

For a fixed time $t\geq t_0=(a^2+b^2+x_{03}^2)^\ha,$ $t>|x_{03}|,$ and since $ x_3=\frac{x_{03}}{t_0} t,$ $t>|x_3|.$    This is consistent with the fact that for fixed $t,$ the three spherical waves intersect along a bounded segment of the hyperbola and the surface of the newly formed wave is bounded for bounded times, see Fig.\ref{wave1}.

Now, with  $x_3$ fixed, and  $t>x_3$ increasing, this is an expanding circular wave centered at $(a,b)$ with radius
\begin{gather*}
R= (t^2-x_3^2)^\ha-(a^2+b^2)^\ha,
\end{gather*}
and therefore $\frac{dR}{dt}=\frac{t}{(t^2-x_3^2)^\ha}\geq 1,$ which will give the appearance that the circular wave is moving faster speed  than the speed of light.

\section{Statements of the Main Results}

In Section \ref{products} below we will define spaces of conormal distributions to a submanifold $\mcm,$ which we shall denote by $I^m(\Omega,\mcm).$    We also discuss products of conormal distributions conormal to transversal hypersurfaces. We shall prove  that if $v_j\in I^m(\Omega,\Sigma_j),$ $j=1,2,3,$  are conormal distributions to $\Sigma_j,$ $j=1,2,3,$ then
\begin{gather}
v_1v_2v_3=  v_T + \sum_{j,k=1}^3 v_{jk} + \sum_{j=1}^3 w_j, \label{prod-st}
\end{gather}
 where $v_T$ is a product-type conormal distribution with respect to the submanifold $\Gamma,$ $v_{jk}$ is a product-type conormal distribution associated to $\Gamma_{jk}$ respectively, and $w_j$ are conormal distributions to $\Sigma_j.$    By product-type conormal distributions, we mean their symbols are product-type symbols which are defined below. The results of Melrose and Ritter \cite{MelRit}, Bony \cite{Bony5,Bony6} and S\'a Barreto \cite{SaB,SaB2} guarantee that if $v$ satisfies the hypothesis of Theorem \ref{triple-0}, then  the solution $u$ of \eqref{Weq} is conormal (in a suitable sense) to $\mcq\cup \Sigma_1\cup \Sigma_2 \cup \Sigma_3.$   As explained above, one has to be careful when applying the results of \cite{Bony5,Bony6,MelRit}.  The new wave $\mcq$ emanating from $\Gamma$ will interact with the three original waves.  However, if one is just interested in showing that after the interaction, the singularities of $u$ will be contained in $\Sigma_1\cup\Sigma_2\cup\Sigma_3\cup\mcq,$ it is enough to assume that $u$ is conormal to $\Sigma_1\cup\Sigma_2\cup\Sigma_3\cup\mcq$ (in a suitable way)  in the past, and show that it the solution remains conormal to the three original surfaces and to $\mcq$ in a neighborhood of $\Gamma.$  This is explained in details in Theorem \ref{REG} below.

 Our first result gives the principal part of the singularity of the solution $u$ in a neighborhood of $\Gamma,$ microlocally near $N^*\mcq\setminus 0$ and away from $N^*\Sigma_j,$ $j=1,2,3.$  In what follows, we define $(\p_u^3 f)(y,u(y))|_{\Gamma}$ to be the restriction of $(\p_u^3 f)(y,u(y))$ to $\Gamma.$  In local coordinates $y=(y',y''),$ $y'=(y_1,y_2,y_3),$ where $\Sigma_j=\{y_j=0\},$ $j=1,2,3,$ and $\Gamma=\{(0,0,0,y")\},$ we have $(\p_u^3 f)(y,u(y))|_{\Gamma}=(\p_u^3 f)(0,0,0,y'',u(0,0,0,y'')).$ 
The distributions in question here are continuous functions, so there is no problem defining this operation. We will also show that $(\p_u^3 f)(y,u(y))|_{\Gamma} v_T$ is invariantly defined modulo smoother terms.

 \begin{theorem}\label{triple-0}  Let $\Omega,$ $P(y,D),$ $t,$ $f(y,u),$ $\Sigma_j\subset  \Omega,$ $1\leq j \leq 3,$ $\Gamma$  
 and $\mcq$ be as defined above.   Let $v_j\in I^{\msi}(\Omega,\Sigma_j),$ $m<-\ha(n+7),$ be conormal distributions to $\Sigma_j,$  satisfying $P(y,D)v_j=0,$ $j=1,2,3,$  in $\Omega.$  Let $u$  be the solution to \eqref{Weq}.  For each $q \in \Gamma$ there exists a neighborhood $\Omega_q$ of $q$ such that microlocally away from $N^*\Sigma_j\setminus 0,$  $j=1,2,3,$  and  $N^*\Gamma\setminus 0,$  $u\in I^{3m-\frac{n}{4}}(\Omega_q,\mcq),$ and $u(y)-E_+\left( \left[(\p_u^3 f)(y,u(y))|_{\Gamma}\right] v_T\right)\in I^{3m-\frac{n}{4}-1}(\Omega_q,\mcq).$
\end{theorem}

Some remarks about Theorem \ref{triple-0}:
\begin{enumerate}[1.]
\item   If  $(\p_u^3f)(y,u)|_{\Sigma}=0,$  where $\Sigma\subset \Gamma$ is a relatively open subset, Theorem \ref{triple-0} does not give any information about the leading order singularities of $u$ on the part of $\mcq$ corresponding to the flow-out of $\Sigma.$  The solution may very well be $C^\infty$ away from the incoming waves.
\item  S\'a Barreto and Wang \cite{SaWang1} proved Theorem \ref{triple-0} in the case where $n=3$ and $f(y,u)$ is a polynomial in $u$ with $C^\infty$ coefficients.
\item  M. Beals \cite{Beals4} proved the local version of this result for $n=3:$ Namely,  the operator $P(y,D)$ has constant coefficients,   $\Sigma_j=\{y_j=0\},$ $v_j=y_{j+}^m,$  $m\in \mn,$   and $f(y,u)=a(y) u^3.$  A modification of the spaces  introduced by Beals in \cite{Beals4} are an important part of this paper and of \cite{SaWang1}.
\end{enumerate}

We also analyze the global behavior of the singularity at $\mcq,$ as long as $\mcq$ remains $\CI.$ The result of Bony \cite{Bony3} on the propagation of conormal singularity along a $C^\infty$ characteristic surface shows that, away from $\Gamma,$ and as long as $\mcq$ remains smooth, $u$ is conormal to $\mcq.$   We want to analyze the evolution of the principal symbol of $u$ along $\mcq.$  This was done by
Rauch and Reed \cite{RauRee1} and Piriou \cite{Pir}, and it is straightforward  for semilinear equations $P(y,D)u=f(y,u).$  Piriou \cite{Pir} studied the evolution of  the principal symbol of solutions of fully nonlinear equations which are conormal to a surface, with the assumption that the surface is a priori known to be $\CI.$   We will prove the following for the convenience of the reader:
 
\begin{prop}\label{REG-D1} Let $\Omega\subset \mr^n$ be open subset and  $P(y,D)$ be a $C^\infty$ second order strictly hyperbolic operator.  Assume that $\Omega$ is bicharacteristically convex with respect to $P(y,D)$ and let $t$ be a time function for $P(y,D).$ Let 
$\Sigma\subset \Omega$ be a $C^\infty$ closed hypersurface that is characteristic for $P(y,D).$ If   
$u$ satisfy \eqref{Weq} and $u\in I^{\mu}(\Omega\cap \{t<0\},\Sigma),$ and $\mu+\novf-\ha<-1,$  then  $u\in I^{\mu}(\Omega,\Sigma).$ Moreover, if $a$ is the principal symbol of $u,$ and $p$ is the principal symbol of $P(y,D),$ then $(\mcl_{H_p} +c)a=0,$ where $\mcl_{H_p}$ is the Lie derivative with respect to the Hamilton vector field $H_p$ and $c$ is the subprincipal symbol of $P(y,D).$
\end{prop}
\begin{proof}  We know that $u\in I^{\mu}(\Omega,\Sigma)$ because of the work of Bony  \cite{Bony3}. We are interested in the evolution of its symbol.

Since $u \in I^{\mu}(\Omega,\Sigma),$ we know from a result of  Rauch and Reed \cite{RauRee1},  see Proposition \ref{hyper-alg} below,  that for $\mu+\novf<-\ha,$ $I^{\mu}(\Omega,\Sigma)$ is a $C^\infty$ algebra and therefore  $P(y,D)u=f(y,u) \in I^{\mu}(\Omega,\Sigma).$ But on the other hand, since $\Sigma$ is characteristic for $P(y,D),$ it follows from Theorem 25.2.4 of \cite{HormanderV4} that 
$P(y,D)u\in I^{\mu+1}(\Omega,\Sigma)$ and its principal symbol is 
equal to $(\mcl_{H_p} +c)a,$ where $a$ is the principal symbol of $u,$ $\mcl_{H_p}$ is the Lie derivative with respect to the Hamilton vector field $H_p$ and $c$ is the subprincipal symbol of $P(y,D).$   But since in fact $P(y,D)u =f(y,u)\in I^{\mu}(\Omega,\Sigma),$ we conclude that $(\mcl_{H_p} +c)a=0.$ This ends the proof of the Proposition.
\end{proof}

Since $\mcq$ is a $C^\infty$ hypersurface away from $\Gamma,$ Proposition \ref{REG-D1} implies the following result regarding the principal symbol of $u$ on the conormal bundle to $\mcq$ as time evolves,  but away from the hypersurfaces and before caustics eventually form on $\mcq:$
\begin{theorem}\label{triple}  Let $\Omega,$ $P(y,D),$ $t,$ $f(y,u),$ $\Sigma_j\subset  \Omega,$ $1\leq j \leq 3,$ $\Gamma,$  
and $\mcq$ be as defined above.   Let $q\in \Gamma$ and let $W\subset \Omega $ be a neighborhood of $q$  such that $\mcq$ is $\CI$ in $W$ and for any open subset $\wtW \subset W$ such that $\wtW\cap \Sigma_j=\emptyset,$ $j=1,2,3,$ 
$u\in I^{3m-\novf}(\wtW, \mcq),$ $3m<-\ha.$  Suppose for a given $\wtW$ there exists another open subset $U$ which is bicharacterisically convex  and  is contained in the domain of influence of $\wtW$ and $U \cap \Sigma_j=\emptyset,$ $j=1,2,3,$  see Fig.\ref{Fig2}.  Suppose that $\mcq\cap U$ is $C^\infty,$ then   $u \in I^{3m-\novf}(U,\mcq)$    and its principal symbol 
$a\in S^\mu(N^*\mcq, \Omega_{N^*\mcq}^\ha)$ satisfies $(\mcl_{H_p}+c)a=0$ on $N^*\mcq\setminus 0,$ where $\mcl_{H_p}$ is the Lie derivative with respect to $H_p,$ and $c$ is the subprincipal symbol of $P.$
\end{theorem}
\begin{figure}
\centering
\psscalebox{.5 .5} 
{
\begin{pspicture}(0,-2.6747327)(9.0111,2.6747327)
\psellipse[linecolor=black, linewidth=0.04, dimen=outer](4.2147365,2.4530485)(4.2,0.2)
\psline[linecolor=black, linewidth=0.04](0.014736328,2.4530485)(0.014736328,2.4530485)(4.4147363,-2.3469515)(8.414737,2.4530485)(8.414737,2.4530485)
\psellipse[linecolor=black, linewidth=0.04, dimen=outer](4.3147364,0.2530484)(2.3,0.2)
\psline[linecolor=black, linewidth=0.04, linestyle=dotted, dotsep=0.10583334cm](3.4147363,2.6530485)(4.4147363,-2.7469516)
\psline[linecolor=black, linewidth=0.05](5.014736,2.2530484)(4.4147363,-2.3469515)
\rput[bl](3.0147364,-0.3469516){$\wtW$}
\rput[bl](5.4147363,0.6530484){$U$}
\psline[linecolor=black, linewidth=0.05](4.923827,1.1439575)(4.014736,-0.032057982)(4.5601907,-1.2080735)(5.4147363,0.053048402)(4.923827,1.1439575)
\rput{-0.790553}(0.0058165262,0.06463736){\psellipse[linecolor=black, linewidth=0.05, dimen=outer](4.6874638,-0.3892307)(1.0545454,0.8988118)}
\psdots[linecolor=black, dotsize=0.04](4.742009,0.07123022)
\psdots[linecolor=black, dotsize=0.04](4.742009,0.07123022)
\psdots[linecolor=black, dotsize=0.04](4.742009,0.07123022)
\psdots[linecolor=black, dotsize=0.08](4.742009,0.07123022)
\psdots[linecolor=black, dotsize=0.08](4.742009,0.07123022)
\psdots[linecolor=black, dotsize=0.08](4.742009,0.07123022)
\psdots[linecolor=black, dotsize=0.1](4.742009,0.07123022)
\psdots[linecolor=black, dotsize=0.1](4.742009,0.07123022)
\psdots[linecolor=black, dotsize=0.1](4.742009,0.07123022)
\psdots[linecolor=black, dotsize=0.1](4.742009,0.07123022)
\rput[bl](7.6511,1.1621393){$\mcq$}
\end{pspicture}
}
\caption{ Propagation of conormality along $\mcq$ from  $\wtW$ and $U$ where $\mcq$ remains smooth and $\wtW$ and $U$ do not intersect the hypersurfaces.}
\label{Fig2}
\end{figure}
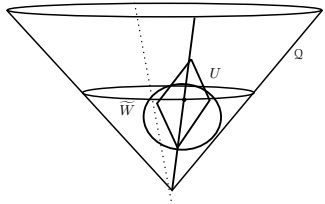

However, eventually $\mcq$ may develop caustics  and Theorem \ref{triple}  is no longer valid.  This can happen for instance if $n=3$ and $P(y,D)=D_t^2-\Delta_g,$ where $\Delta_g$ is the Laplacian with respect to a Riemannian metric $g.$ If the metric $g$ has conjugate points, geodesics emanating from one point (the tip of the cone) intersect at another point, and this causes a singularity on the cone $\mcq,$ see Fig.\ref{Fig3}. There are results on the propagation of conormal singularities for smilinear wave equations in the presence of caustics. The case of the cusp caustic was studied by Melrose \cite{Melrose1,Melrose2} and Beals \cite{Beals2}, the swallowtail caustic was studied by Delort \cite{Delort1},  Joshi and S\'a Barreto \cite{JosSab}, Lebeau \cite{Lebeau1} and S\'a Barreto \cite{SaB1}. Beals \cite{Beals3} and Melrose and S\'a Barreto \cite{MelSab} analyzed the case of the interaction of a cusp and a plane.

 The remainder of the paper is divided in three sections.  In Section \ref{sketch} we outline the main ideas of the proof of Theorem \ref{triple-0}.  In Section \ref{products} we introduce the necessary spaces distributions and we recall the results of Melrose and Ritter \cite{MelRit},  Bony \cite{Bony5,Bony6} and S\'a Barreto \cite{SaB,SaB2} about the  interactions of conormal waves and in section \ref{gen-sing} we prove Theorem \ref{triple-0}.  
 
 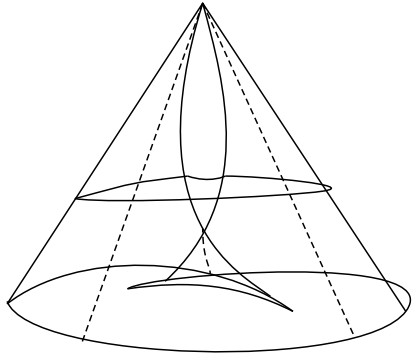
\begin{figure}[h!]
\centering
\psscalebox{.5 .5} 
{
\begin{pspicture}(0,-4.678309)(10.796289,4.678309)
\psline[linecolor=black, linewidth=0.04](10.618725,-3.55828)(5.218725,4.6417203)(0.018725282,-3.35828)
\psbezier[linecolor=black, linewidth=0.04](0.018725282,-3.291613)(0.6187253,-4.891613)(9.418725,-5.15828)(10.618725,-3.5582798767089843)(11.818726,-1.9582798)(3.8187253,-2.55828)(3.2187252,-2.9582798)
\psbezier[linecolor=black, linewidth=0.04](3.2187252,-2.9582798)(4.87662,-2.7182798)(6.197673,-2.83828)(7.6187253,-3.5582798767089843)
\psbezier[linecolor=black, linewidth=0.04](0.018725282,-3.35828)(1.8187252,-1.9582798)(5.4187255,-1.9582798)(7.6187253,-3.5582798767089843)
\psbezier[linecolor=black, linewidth=0.04](5.218725,4.6417203)(3.4187253,-1.1582799)(6.218725,-2.55828)(7.0187254,-3.1582798767089844)
\psbezier[linecolor=black, linewidth=0.04](5.218725,4.6417203)(5.818725,2.44172)(6.6187253,-0.5582799)(4.218725,-2.7582798767089844)
\psbezier[linecolor=black, linewidth=0.04, linestyle=dashed, dash=0.17638889cm 0.10583334cm](5.218725,-1.3582798)(5.218725,-2.15828)(5.4187255,-2.35828)(5.4187255,-2.5582798767089843)
\psline[linecolor=black, linewidth=0.04, linestyle=dashed, dash=0.17638889cm 0.10583334cm](5.218725,4.6417203)(2.0187254,-4.3582797)
\psline[linecolor=black, linewidth=0.04, linestyle=dashed, dash=0.17638889cm 0.10583334cm](5.218725,4.6417203)(9.218725,-4.15828)
\psbezier[linecolor=black, linewidth=0.04](1.8187252,-0.5582799)(2.807097,-0.71033704)(7.429789,-0.50662035)(8.418725,-0.35827987670898437)(9.407661,-0.20993942)(6.818725,0.041720122)(5.818725,0.041720122)
\psbezier[linecolor=black, linewidth=0.04](1.8187252,-0.5582799)(3.6069605,-0.078279875)(3.2187252,-0.15827988)(4.818725,0.041720123291015626)
\psbezier[linecolor=black, linewidth=0.04](4.818725,0.041720122)(5.4187255,-0.15827988)(5.818725,0.041720122)(5.818725,0.041720123291015626)
\end{pspicture}
}
\caption{ A swallowtail singularity formed in the light cone emanating from a point in two space dimensions.  This can be due  to the existence of conjugate points of the geodesic flow in the case $P=D_t^2-\Delta_g,$ $g$ a Riemannian metric in $\mr^{2}.$  The solution to \eqref{Weq} would remian conormal to $\mcq$ away from the caustic, but other singularities could be generated by the caustic. This figure resembles one after equation 5.1.24 in Duistarmaat's book \cite{Duist}.}
\label{Fig3}
\end{figure}

\section{Outline of the Proof of Theorem \ref{triple-0}}\label{sketch}

  We explain the main ideas of the proof in a simplified version of the theorem. Here we shall suppose that $\Omega$ is small enough that there exist local coordinates 
   \begin{gather*}
   \begin{gathered}
   y=(y',y''), y'=(y_1,y_2,y_3)  \text{ such that } \Sigma_j=\{y_j=0\}, \ j=1,2,3, \\
    \end{gathered} 
   \end{gather*}
 valid in $\Omega.$    We  denote $F(y,u)=\mcy(y) f(y,u)$ and analyze the singularities of the solution of 
\begin{gather}
\begin{gathered}
P(y,D)u= F(y,u),  \\
u=v_1+v_2+v_3 \text{ for } t<-1.
\end{gathered}\label{Weq1}
\end{gather}
  
  If  $s>\novt,$ it is well known that $H_{\loc}^s(\Omega)$ is a $C^\infty$ algebra -- it is closed under composition with $C^\infty$ functions.   If $v\in H_{\loc}^s(\Omega),$ equation \eqref{Weq}  can be solved by  using a contraction mapping argument to show that, for small enough $\Omega,$ there exists a unique
 $u\in H_{\loc}^s(\Omega)$ that satisfies
\begin{gather}
u= v + E_+( F(y,u)), \label{it1}
\end{gather}
where $E_+$ is the forward fundamental solution of $P$ and
\begin{gather*}
E_+: H_{\loc}^s (\Omega) \longrightarrow H^{s+1}_{\loc}(\Omega).
\end{gather*}

 Now to analyze the propagation of singularities, we assume that the solution exists and proceed as in \cite{Beals4} and \cite{SaWang1}. We take advantage of the fact that $E_+( F(y,u))$ is smoother than $u,$ we  iterate this formula and obtain
\begin{gather*}
u= v+ E_+[ F(y, v+ E_+(F(y,u)))].
\end{gather*}

We  shall appeal to the work of Rauch and Reed \cite{RauRee1} and Piriou \cite{Pir} which show that that if $v_j \in I^{\msi}(\Omega,\Sigma_j)$ with  $m<-1,$  and $\Sigma_j=\{y_j=0\},$ then $v_j= \nu_j+\mce_j,$ where $\nu_j= y_j^{k} w_j,$  $k=k(m)$ is the non-negative integer such that $-m-2 \leq k(m) < -m-1,$ $\mce_j\in C^\infty$ and $w_j\in I^{\msi+k(m)}(\Omega,\Sigma_j).$ We then write
\begin{gather}
\begin{gathered}
v=v_1+v_2+v_3=\nu+\mce, \;\ \nu= \nu_1+\nu_2+\nu_3, \;\ \mce\in C^\infty, \\
\mcw= \mce+ E_+(F(y,u)), \text{ and so } 
u= \nu+\mce + E_+[ F(y, \nu(y)+ \mcw)]. 
\end{gathered}\label{defW}
\end{gather}
Since $\nu=0$ at $\Gamma,$ $\mcw=u$ at $\Gamma.$ We then expand $F(y, \nu(y)+\mcw)$ in Taylor series in $\nu$ centered  at $\mcw:$
\begin{gather*}
F(y, \nu(y)+ \mcw)= F(y,\mcw)+ (\p_u F)(y,\mcw) \nu + \ha (\p_u^2 F)(y,\mcw) \nu^2 +  \frac{1}{6} (\p_u^3 F)(y,\mcw) \nu^3 + \\
\frac{\nu^4}{6} \int_0^1 (\p_u^4 F)(y, t\mcw + (1-t)\nu)(1-t)^3 \, dt.
\end{gather*}
 We will work introduce a variation of  the spaces introduced by Beals \cite{Beals4},  see Definition \ref{BSDEF}. These spaces will be used to filter singularities and to show that:\\
{\bf Claim 1:}   The term 
\begin{gather*}
\mcr(y)=F(y,\mcw)+ (\p_u F)(y,\mcw) u + \ha (\p_u^2 F)(y,\mcw) \nu^2 + 
\frac{\nu^4}{6} \int_0^1 (\p_u^4 F)(y, t\mcw + (1-t)\nu)(1-t)^3 \, dt \\
\text{ is smoother than } \frac{1}{6} (\p_u^3 F)(y,\mcw) \nu^3.
 \end{gather*}
{\bf Claim 2:}  We will show that,  
 \begin{gather*}
 \frac{1}{6} (\p_u^3 F)(y,\mcw) \nu^3=   (\p_u^3 F)(y,\mcw) (\nu_1\nu_2\nu_3) + \text{{smoother terms.}  }
 \end{gather*}

However,  the term  $(\p_u^3 F)(y,\mcw) (\nu_1\nu_2\nu_3)$
 still does not say vey much about the singularities of $u$ because of course $\mcw$ depends on $u.$\\
{\bf Claim 3:}  We will show that,
\begin{gather*}
(\p_u^3 F)(y,\mcw) (\nu_1\nu_2\nu_3)= (\p_u^3 F)(0,0,0,y'',\mcw(0,0,0,y'')) (\nu_1\nu_2\nu_3) + \text{ smoother terms.}
\end{gather*}
But according to \eqref{defW}, $\mcw(0,0,0,y'')=u(0,0,0,y'').$

 \section{Spaces of Distributions}\label{products}

 We first recall the definition of the class of conormal distributions to a $\CI$ closed submanifold $\mcm \subset \Omega$ of codimension $k.$  Let $\mcv_\mcm$ denote the Lie algebra of $\CI$ vector fields tangent to $\Sigma.$

 As in  H\"ormander \cite{HormanderV3},  we say that  $u\in I^{m}(\Omega, \mcm),$ $m\in \mr,$ and $u$ is a conormal  distribution to $\mcm$ of order $m,$  if for any $N\in \mn,$ 
 \begin{gather*}
V_1 V_2 \ldots V_N u \in {}^\infty H_{-m-\frac{n}{4}}^{\loc} (\Omega), \;\ V_j \in \mcv_\mcm.
\end{gather*}
  The definition of ${}^\infty H_{-m-\frac{n}{4}}^{\loc} (\Omega)$ can be found in Appendix B of \cite{HormanderV3}.   It follows from the definition of Besov spaces that
\begin{gather}
  I^{m}(\Omega, \mcm) \subset I^{m'}(\Omega, \mcm), \;\ m\leq m'.\label{inc-conormal}
\end{gather}

  In general, if $\mcw$ is a Lie algebra and $\CI$ module of $\CI$ vector fields, the space of conormal distributions with respect to 
$\mcw$ is defined to be 
\begin{gather}
I H_{\loc}^{s}(\Omega,\mcw)=\{u \in H_{\loc}^s(\Omega): V_1 V_2 \ldots V_N u \in H_{\loc}^s(\Omega), \; s\in \mr,  V_j \in \mcw, \ N \in \mn\}.
\label{GEN-conormal}
\end{gather}
This can also be defined in terms of  Besov spaces, instead of Sobolev ones. We will consider spaces of conormal distributions related to the interaction of waves.

 Let $\Sigma_j\subset \Omega$ $j=1,2,3$ be closed $C^\infty$ hypersurfaces which intersect transversally at $\Sigma_i\cap \Sigma_j=\Gamma_{ij}$ $i\not=j,$ and  at $\Gamma=\Sigma_1\cap \Sigma_2\cap \Sigma_3.$ Let $\mcq$ be defined as above. The following Lie algebras and  $\CI$ modules of $C^\infty$ vector fields will play an important role in this paper:
\begin{gather}
\begin{gathered}
\mcw_j,  \text{ denotes the } \CI \text{ vector fields tangent to } \Sigma_j, \; j=1,2,3, \\
\mcw_{jk}  \text{ denotes the } \CI  \text{ vector fields tangent to } \Sigma_j \cup \Sigma_k, \; j=1,2,3, \; j\not=k, \\
\mcw_{123} \text{  denotes the } \CI \text{ vector fields tangent to } \Sigma_1 \cup \Sigma_2 \cup \Sigma_3, \\
\mcw_{123,\mcq} \text{ denotes the } \CI \text{ vector fields tangent to } \Sigma_1\cup \Sigma_2 \cup \Sigma_3\cup \mcq.
\end{gathered}\label{123Q}
\end{gather}
These Lie algebras are locally finitely generated. In local coordinates $y=(y',y''),$ $y'=(y_1,y_2,y_3),$ where $\Sigma_j=\{y_j=0\},$ $j=1,2,3,$ we have
\begin{gather}
\begin{gathered}
\mcw_j= \CI-\text{ span of } \{y_j\p_{y_j}, \p_{y_k}, \ k\not=j\}, \\
\mcw_{jk}= C^\infty(\Omega)-\text{span of } \{ y_j\p_{y_j}, y_k\p_{y_k},  \p_{y_m}, \;  m\not= j,k, \}, \\
\mcw_{123}= C^\infty(\Omega)-\text{span of } \{ y_1\p_{y_1}, y_2\p_{y_2},  y_3\p_{y_3}, \p_{y_m}, \; \; m\geq 4\}.
\end{gathered}\label{span}
\end{gather}
$\mcw_{123,\mcq}$ is also finitely generated, see \cite{MelRit}.
  
  The class of symbols  $S^{r}(\mr^n\times \mr^k)$  is defined as the space of $\CI(\mr^n\times \mr^k)$  functions that satisfy
\begin{gather}
|D_{y}^\alpha D_{\eta'}^\beta b(y,\eta')| \leq C_{\alpha,\beta}(1+|\eta'|)^{r-|\beta|}, \;\ \alpha\in \mn^{n},  \beta\in \mn^k. \label{symb-def}
\end{gather}
These spaces satisfy
\begin{gather}
S^r(\mr^{n-k}\times \mr^{k}) \subset S^{r'}(\mr^{n-k}\times \mr^{k}), \;\ r\leq r'. \label{symb-incl}
\end{gather}

   Theorem 18.2.8 of \cite{HormanderV3} says that $u\in I^{m}(\Omega, \mcm)$ if and only if  $u \in C^\infty(\mr^{n+1}\setminus \mcm)$ and in a neighborhood of any point $p\in \mcm,$ in local coordinates where 
\begin{gather}
y=(y', y''),\  y'=(y_1, y_2, \ldots, y_k), \text{ such that } \mcm=\{y_1=y_2=\ldots= y_k=0\}, \label{defyp}
\end{gather}
$u$ is given by
\begin{gather}
u(y)= \int_{\mr^k} e^{i  y'\cdot \eta' } a(y,\eta') \; d\eta', \;\ a\in S^{m+\frac{n-2k}{4}}( \mr_y^n \times \mr^{k}_{\eta'}). \label{con1}
\end{gather}

If one multiplies a conormal distribution  $u\in I^m(\Omega,\mcm)$ by a $\CI$ function $f$ which vanishes on $\mcm,$ $f u\in I^{m'}(\Omega,\mcm),$ with $m'<m.$ This is made precise in  the following
\begin{prop}\label{xtoalpha}(Proposition 18.2.3 of \cite{HormanderV3}). Let $\mcm\subset \Omega$ be a $C^\infty$ submanifold of codimension $k.$ Let $u \in I^{m}(\Omega,\mcm)$ and let $y=(y',y''),$  be local coordinates as in \eqref{defyp}. If  $\alpha\in \mn^k,$  then ${y'}^\alpha u \in I^{m-|\alpha|}(\Omega,\mcm).$
\end{prop}
Therefore, if $u\in I^m(\Omega,\mcm)$  satisfies \eqref{con1}, and $y=(y',y''),$ satisfy \eqref{defyp}, its Taylor expansion about $\{y'=0\}$ satisfies
\begin{gather*}
a(y',y'',\eta')- \sum_{|\alpha|\leq k} \frac{1}{\alpha!} {y'}^\alpha \p_{y'}^\alpha a(0, y'',\eta')= O(|y'|^{k+1}),
\end{gather*}
and therefore, by Borel summation formula, 
\begin{gather}
\begin{gathered}
u(y)= \int_{\mr^k} e^{i y'\cdot \eta'} a(0,y'',\eta') d\eta' +  \int_{\mr^k} e^{i y'\cdot \eta'} b(0,y'',\eta') d\eta'+ \mce, \text{ where } \mce \in \CI, \text{ and } \\
b(0,y'',\eta')\sim \sum_{|\alpha|\geq 1} \frac{i^{|\alpha|}}{\alpha!} \p_{y'}^\alpha \p_{\eta'}^\alpha a(0,y'',\eta')\in S^{m-1+\frac{n-2k}{4}}( \mr_y^n \times \mr^{k}_{\eta'}).
\end{gathered}\label{left-red}
\end{gather}
The principal symbol of $u$ is defined to be
\begin{gather*}
[a(0,y'',\eta')] \in S^{m+\frac{n-2k}{4}}(\mr^{n-k} \times \mr^k)/S^{m-\frac{n-2k}{4}-1}(\mr^{n-k} \times \mr^k),
\end{gather*}
which is the  equivalence class of $a(0,y'',\eta')$ in this quotient. 

 However, this definition is not coordinate invariant. Following \cite{HormanderV3},  this issue is resolved if one thinks of conormal distributions as distributions acting on half-densities $\Gamma_\Omega^\ha$ and their principal symbol as an element of the half-density bundle $\Gamma^\ha_{N^*\mcm}$ on the conormal bundle $N^*(\mcm).$ In local coordinates $(y',y'',\eta',\eta'')$ this is given by
\begin{gather}
a(0,y'',\eta')|dy''|^\ha|d\eta'|^\ha \in S^{m+\frac{n}{4}}( N^*\mcm,\Gamma_{N^*\mcm}^\ha). \label{pr-symb}
\end{gather}
 One needs to realize that $|d\eta'|^\ha$ is homogeneous of  degree $\frac{k}{2},$ and so this is a symbol of the order stated above. 

We conclude that the principal symbol map is the isomorphism
\begin{gather*}
S^{m+\frac{n}{4}}(N^*\mcm,\Gamma_{N^*\mcm}^\ha)/ S^{m+\frac{n}{4}-1}(N^*\mcm,\Gamma_{N^*\mcm}^\ha) \longrightarrow I^{m}(\Omega, \mcm; \Gamma_{\Omega}^\ha)/ I^{m-1}(\Omega, \mcm; \Gamma_{\Omega}^\ha) \\
 [a] \longmapsto [u].
 \end{gather*}

\subsection{Further Properties of Conormal Distributions}   First, we recall a result due to Rauch and Reed, Proposition 2.1 of \cite{RauRee1} which is very important in the study of nonlinear equations: 
  \begin{prop}\label{hyper-alg} Let $\Sigma\subset \Omega$ be a closed $C^\infty$ hypersurface, and let $u_j,$ $j=1,2,\ldots,N$ be a family of distributions in $I^{\msi}(\Omega,\Sigma),$  with $m<-1.$ If   $f(y,x_1,\ldots, x_n)\in C^\infty,$  then $f(y,u_1,\ldots, u_N) \in   I^{\msi}(\Omega,\mcm).$
  \end{prop}

Next we recall properties of conormal distributions established by Rauch and Reed \cite{RauRee1} and Piriou \cite{Pir}.  
Let $u\in I^{\msi}(\Omega,\Sigma),$  $m<-1$ where $\Sigma$ is a $C^\infty$ closed hypersurface on $\Omega$ and  in local coordinates \eqref{defyp}, $\Sigma=\{y_1=0\}.$ Then 
 \begin{gather*}
u(y)= \frac{1}{2\pi} \int_\mr e^{i y_1 \eta_1} a(y'',\eta_1) d\eta_1+ \mce, \; a\in S^m(\mr^{n-1}\times \mr ), \;\ \mce\in C^\infty.
\end{gather*}
One can show that if  $m<-1,$  and $k(m)$ is the  non-negative integer such that $-m-2\leq k(m) < -m-1,$ then
by modifying the symbol of $u$ on a compact set in $\eta_1$ we have
\begin{gather}
\begin{gathered}
u(y) =  y_1^k v_k(y) + \mce, \text{ such that } \ \mce\in C^\infty,  \text{ and } \\
v_k(y)=\int_{\mr} e^{i y_1\eta_1} b_k(y'',\eta_1) d\eta_1,  \; b_k\in S^{\msi+k},  \ 0\leq k\leq k(m). 
\end{gathered}\label{van-u}
\end{gather}
see for example \cite{Pir}, or \cite{SaWang1} details. We then define, as in \cite{Pir,RauRee1}, 
\begin{definition} Let  $\Sigma \subset \Omega$ be a $C^\infty$ closed hypersurface. For $m<-1,$ the space $\idm(\Omega,\Sigma)$ consisting of elements $u\in I^{\msi}(\Omega,\mcm)$ which  in local coordinates where $\Sigma=\{y_1=0\},$  can be written as $u= y_1^k v_k,$ with $v_k\in I^{\msi-k}(\Omega,\mcm),$  $k\leq k(m).$ As above, $k(m)$ is the only positive integer in the interval $[-m-2, -m-1).$  
\end{definition}

And one can prove, see \cite{Pir,RauRee1,SaWang1}
\begin{prop}\label{decomp-01}  Let $\Sigma \subset \Omega$ be a $C^\infty$ closed hypersurface, and let $u\in I^{\msi}(\Omega,\mcm),$ with $m<-1,$ then
\begin{gather}
u(y)=v(y) +\mce, \;\ v \in \idm(\Omega,\Sigma), \;\ \mce \in C^\infty. \label{van-u1}
\end{gather}
\end{prop}

As a consequence of the definition of $\idm(\Omega,\Sigma)$ and Proposition \ref{hyper-alg} we have:
\begin{prop}\label{prod1}  Let $\mcm\subset \Omega$ be a closed $C^\infty$ hypersurface and let $u_j \in \idm(\Omega,\mcm),$ $1\leq j \leq N,$ and $m<-1.$ Then $u_1u_2\ldots u_N \in I^{\msi-(N-1)k(m)}(\Omega,\mcm).$ In particular, if $u \in \idm(\Omega,\mcm),$ then $u^N \in  I^{m-(N-1)k(m)-\novf+\ha}(\Omega,\mcm)$ 
\end{prop}

 We introduce spaces of distributions that will be used in the proof of Theorem \ref{triple-0}.

\subsection{ Conormal Distributions Associated with  Double and Triple Interactions}  We briefly recall the results of  Bony \cite{Bony4,Bony5} and Melrose and Ritter \cite{MelRit} about the evolution of one wave and  the double and triple transversal interactions.

Bony \cite{Bony3,Bony4} proved the following result regarding  the propagation of conormal regularity with respect to one hypersurface and two transversal hypersurfaces, see also \cite{MelRit}:
\begin{theorem}\label{REGSD}  Let $u \in H_{\loc}^s(\Omega),$ $s>\novt,$ satisfy $P(y,D) u= f(y,u),$ $f\in \CI.$  Let $\Sigma_1$ and $\Sigma_2$ be closed $\CI$ hypersurfaces in $\Omega$ intersecting transversally, and let $\mcw_1$ and $\mcw_{jk}$ be the Lie algebras of $\CI$ vector fields defined in \eqref{123Q}:
 \begin{enumerate}[1.]
 \item If  $u\in IH_{\loc}^s(\Omega,\mcw_{1})$ in $t<0,$ then   $u \in I H_{\loc}^{s}(\Omega, \mcw_{1}).$
 \item If  $u\in IH_{\loc}^s(\Omega,\mcw_{12})$ in $t<0,$ then  $u \in I H_{\loc}^{s}(\Omega, \mcw_{12}).$
 \end{enumerate}
\end{theorem}

Let $\mcw_{123}$ denote the Lie algebra of $\CI$ vector fields tangent to $\Sigma_j,$ $j=1,2,3.$   The purpose of this paper is to show that in general
\begin{gather*}
\text{ if  } u\in IH_{\loc}^s(\Omega,\mcw_{123}) \text{ in } t<0, \text{ then  } u\not\in I H_{\loc}^{s}(\Omega, \mcw_{123}),
\end{gather*}
since singularities will form on $\mcq.$ One may ask whether
\begin{gather*}
\text{ if  } u\in IH_{\loc}^s(\Omega,\mcw_{123\mcq}) \text{ in } t<0 \text{ implies that  } u\not\in IH_{\loc}^{s}(\Omega, \mcw_{123\mcq}),
\end{gather*}
where $\mcw_{123\mcq}$ denotes the Lie algebra of $\CI$ vector fields tangent to $\Sigma_j,$ $j=1,2,3,$ and $\mcq.$ The answer is not known to the author, because the Lie algebra $\mcw_{123\mcq}$ is too degenerate at $\Gamma,$ but one can construct smaller spaces which coincide with $IH_{\loc}^{s}(\Omega\setminus \Gamma, \mcw_{123\mcq}),$ locally in $\Omega\setminus \Gamma,$ but with vector fields that are less degenerate  at $\Gamma,$ that do propagate.  This can be done by blowing-up $\Gamma$ as in \cite{MelRit,SaB} or by using second microlocalization as in \cite{Bony5,Bony6}.   As in \cite{SaB,SaB2}, one can construct a space of distributions denoted by $J(\Omega)$ which satisfies the following
\begin{enumerate}[1.]
\item $IL_{\loc}^2(\Omega,\mcw_{123})\subset J(\Omega) \subset IL^2_{\loc}(\Omega,\mcw_{123\mcq}),$ 
\item If $Pu=f(y,u)$  and $u\in IL^2_{\loc}(\Omega\cap \{ |t|> \eps\},\mcw_{123\mcq}),$  then $u\in J(\Omega\cap \{|t|>\eps\})$   for any $\eps>0,$
\end{enumerate} 
and such that
\begin{theorem}\label{REG}  Let $u \in H_{\loc}^s(\Omega),$ $s>\novt,$ satisfy $P(y,D) u= f(y,u),$ $f\in C^\infty.$  Let $\Sigma_j,$ $j=1,2,3$ be closed $C^\infty$  characteristic hypersfurfaces intersecting transversally at $\Gamma$ and let $\mcq$ be as defined above.  If $u\in J(\Omega\cap \{t<0\})$ then $u\in J(\Omega).$   
\end{theorem}
The space $J(\Omega)$ is defined using blow-up techniques of Melrose and Ritter \cite{MelRit} and microlocal methods \cite{SaB}. We do not define it to avoid a lengthy discussion that would not be relevant to the rest of the paper, and refer the reader to \cite{MelRit,SaB,SaB2} for more details.

\subsection{Other Spaces Associated with the Triple Interaction}\label{Bspaces}   In this section we introduce a generalization of the spaces defined by  Beals \cite{Beals4}, which he used to  prove Theorem \ref{triple-0} in the particular case where $n=3,$ $f(y,u)= a(y) u^3,$ $P(y,D)$ has constant coefficients and the initial data is $v_j=y_{j+}^m,$ $m\in \mn.$   The  spaces  introduced by Beals were also used by S\'a Barreto and Wang \cite{SaWang1} to analyze the formation of singularities in the triple interaction when $n=3,$ $P(y,D)u = f(y,u),$ where $f(y,u)$ is a polynomial of arbitrary degree in $u$ with $\CI$ coefficients.

 First we need to prove that one can always choose local coordinates $y=(y_1,y_2,y_3,y'')$ such that the surfaces $\Sigma_j,$ $j=1,2,3$ and  the operator $P(y,D)$ are simultaneously put in normal form.  A similar result was proved for $n=3$ in \cite{SaWang1}.
\begin{theorem}\label{loc-coord} Let $\Omega\subset \mr^n$ be an open subset, let $P(y,D)$ be a second order strictly hyperbolic operator in $\Omega.$ Let $\Sigma_j\subset \Omega,$ $j=1,2,3,$ be closed $C^\infty$ hypersurfaces that are characteristic for $P(y,D)$ and intersect transversally at $\Gamma_{jk}=\Sigma_j\cap \Sigma_k$ and at $\Gamma=\Sigma_1\cap\Sigma_2\cap \Sigma_3.$   We have the following normal forms for $\Sigma_j$ and $P(y,D):$
\begin{enumerate}[NF.1]
\item\label{NF1} If  $q\in \Sigma_1\setminus (\Sigma_2\cup \Sigma_3),$ there exist local coordinates $y=(y_1,y'')$ near $q$ such that
\begin{gather}
\begin{gathered}
\Sigma_1=\{y_1=0\} \text{ and } \\
P(y,D)= \sum_{j=2}^n b_{1j}(y) \p_{y_1}\p_{y_j} + \sum_{j,k=2}^n b_{jk}(y) \p_{y_j}\p_{y_k} + \mcl,
\end{gathered}\label{norm-form-1h}
\end{gather}
where $b_{1j}\in \CI,$  $a_{jk}\in \CI$ and  $\mcl$ is a differential operator of order one.  Similar formulas hold near $\Sigma_j,$ $j=2,3.$ \\
\item For  $q\in (\Sigma_1\cap \Sigma_2)\setminus \Sigma_3,$ there exist local coordinates $y=(y_1, y_2,y'')$ near $q$ such that
\begin{gather}
\begin{gathered}
\Sigma_j=\{y_j=0\}, \;\ j=1,2, \text{ and } \\
P(y,D)= b_{12}(y) \p_{y_1}\p_{y_2} + \sum_{j=3}^n b_{1j}(y) \p_{y_1}\p_{y_j} + \sum_{j=3}^n b_{2j}(y) \p_{y_2}\p_{y_j} + 
\sum_{j,k=3}^n b_{jk}(y) \p_{y_j}\p_{y_k} + \mcl,
\end{gathered}\label{norm-form-2h}
\end{gather}
where $b_{12}, a_{jk}\in \CI,$ $b_{12}\not=0$ near $\Gamma_{12},$ and $\mcl$ is a differential operator of order one.   Similar formulas hold near $(\Sigma_1\cap \Sigma_3) \setminus \Sigma_2$ and 
near $(\Sigma_2\cap \Sigma_3) \setminus \Sigma_1.$ \\
\item For $q\in \Gamma$ there exist coordinates $y=(y_1,y_2,y_3,y'')$ valid in a neighborhood $U$ of $q$   such that
\begin{gather}
\begin{gathered}
\Sigma_j=\{y_j=0\},  \;\ j=1,2,3 \text{ and } \\
P(y,D)= b_{12}(y) \p_{y_1}\p_{y_2} + b_{13}(y) \p_{y_1}\p_{y_3}+ b_{23}(y) \p_{y_2}\p_{y_3}+ \sum_{j=1, k=4}^n b_{jk}(y) \p_{y_j}\p_{y_k}+
\mcl,
\end{gathered}\label{loc-coord-1}
\end{gather}
where $\mcl$ is a differential operator of order one.   Since $P(y,D)$ is strictly hyperbolic,  
$b_{jk}\not=0,$ $j,k=1,2,3,$ near $\Gamma.$ 
\end{enumerate}
\end{theorem}
\begin{proof}  The proof of the last case contains the proofs of the other two cases, and we will concentrate on it.  We start by choosing  coordinates $Y=(Y_1,Y_2,Y_3,Y'')$  near $q\in \Gamma$ such that $\Sigma_j=\{Y_j=0\},$ $j=1,2,3.$ This can be done because the surfaces intersect transversally. Since $\Sigma_j,$ $j=1,2,3,$  are characteristic for $P(Y,D),$ we must have
\begin{gather*}
P(Y,D)= a_{11}(Y) Y_1\p_{Y_1}^2+ a_{12}(Y) \p_{Y_1}\p_{Y_2} + a_{13}(Y) \p_{Y_1}\p_{Y_3}+ a_{22}(Y) Y_2 \p_{Y_2}^2+ a_{23}(Y) \p_{Y_2}\p_{Y_3}+ \\  a_{33}(Y) Y_3 \p_{Y_3}^2+ \sum_{j=1, k=4}^n a_{jk}(Y) \p_{Y_j}\p_{Y_k}+
\mcl.
\end{gather*}
We want to find a change of variables $Y= \Psi(Y)$ which preserves the hypersurfaces $\Sigma_j,$ $j=1,2,3,$ such that \eqref{loc-coord-1} holds. We must  have
\begin{gather*}
y_j = Y_j X_j(Y), \ j=1,2,3, \;  |X_j(Y)|>0 \text{ near } 0, \;\ y_j= W_j(Y), \; j\geq 4
\end{gather*}
and therefore,
\begin{gather}
\begin{gathered}
\p_{Y_1}= (X_1+ Y_1\p_{Y_1} X_1) \p_{y_1}+ Y_2\p_{Y_1} X_2 \p_{y_2}+ Y_3\p_{Y_1}X_3 \p_{y_3}, \\
\p_{Y_2}= Y_1\p_{Y_2}X_1 \p_{y_1}+ (X_2+ Y_2\p_{Y_2} X_2) \p_{y_2}+ Y_3\p_{Y_2}X_3 \p_{y_3},  \\
\p_{Y_3}= Y_1\p_{Y_3}X_1 \p_{y_1}+ Y_2\p_{Y_3}X_2 \p_{y_2}+ (X_3+ Y_3\p_{Y_3} X_3) \p_{y_3}, \\
\p_{Y_k}= Y_1 \p_{Y_k} X_1 \p_{y_1}+Y_2 \p_{Y_k} X_2 \p_{y_2}+Y_3 \p_{Y_k} X_3 \p_{y_3}+ \sum_{j=4}^n \p_{Y_k} W_j \p_{y_k} \;\ 4\leq k \leq n.
\end{gathered}\label{push-F}
\end{gather}
Therefore \eqref{loc-coord-1} transforms into
\begin{gather*}
P(y,D) =  \frac{\zed_1}{X_1} y_1\p_{y_1}^2 + \frac{ \zed_2}{Y_2} y_2\p_{y_2}^2+  \frac{\zed_3}{X_3} y_3 \p_{y_3}^2 + A_{12} \p_{y_1}\p_{y_2} + A_{13} \p_{y_1}\p_{y_3}+ \\ A_{23} \p_{y_2}\p_{y_3} + \sum_{j=1, k=4}^n A_{jk}(Y) \p_{y_j}\p_{y_k}+ \wt{\mcl}(y,\p_y),
\end{gather*}
 where  $\wt{\mcl}$ is a differential operator of order one. 
 
 Let  $\Theta_j= X_j+ Y_j \p_{Y_j} X_j,$ then
 \begin{enumerate}[Term.1]
 \item:  $\zed_1$ satisfies
 \begin{gather*}
\zed_1(Y,X_1, \nabla_Y X_1)=   a_{11} \Theta_1^2+ 
a_{12} \Theta_1 \p_{Y_2} X_1+ a_{13}\Theta_1 \p_{Y_3} X_1+  
Y_1Y_2 a_{22} (\p_{Y_2} X_1)^2+ 
 Y_1 a_{23}\p_{Y_2}X_1\p_{Y_3} X_1+ \\ Y_1Y_3 a_{33} (\p_{Y_3} X_1)^2 
 +  \sum_{k=4}^n(a_{1k}\Theta_1+ Y_1a_{2k}\p_{Y_2}X_1 + Y_1 a_{3k}\p_{Y_3}X_1)\p_{Y_k} X_1
+Y_1 \sum_{j,k=4}^n a_{jk}\p_{Y_j}X_1 \p_{Y_k} X_1.
\end{gather*}
\item: $\zed_2$ satisfies
\begin{gather*}
\zed_2(Y,X_2, \nabla_Y X_2)=   Y_1 Y_2 a_{11} (\p_{Y_1} X_2)^2+  a_{12} \Theta_2 \p_{Y_1} X_2+  Y_2 a_{13}\p_{Y_1} X_2 \p_{Y_3} X_2 + a_{22} \Theta_2^2 + a_{23} \Theta_2 \p_{Y_3}X_2+ \\ Y_2 Y_3 a_{33} (\p_{Y_3} X_2)^2+ 
  \sum_{k=4}^n(a_{2k}\Theta_2+ Y_2a_{1k}\p_{Y_1}X_2 + Y_2a_{3k}\p_{Y_3}X_2)\p_{Y_k} X_2
+Y_1 \sum_{j,k=4}^n a_{jk}\p_{Y_j}X_2 \p_{Y_k} X_2.
\end{gather*}
\item: $\zed_3$ satisfies
\begin{gather*}
\zed_3(Y,X_3, \nabla_Y X_3)= Y_1 Y_3 a_{11} (\p_{Y_1} X_3)^2 + Y_3 a_{12} \p_{Y_1}X_3\p_{Y_2}X_3+
a_{13} \Theta_3 \p_{Y_1} X_3+  Y_2Y_3 a_{22} (\p_{Y_2} X_3)^2 +\\ a_{23} \Theta_3 \p_{Y_2} {X_3}+  
a_{33} \Theta_3^2+   
  \sum_{k=4}^n(a_{3k}\Theta_3+ Y_3a_{1k}\p_{Y_1}X_3 + Y_3 a_{2k}\p_{Y_2}X_3)\p_{Y_k} X_3
+Y_1 \sum_{j,k=4}^n a_{jk}\p_{Y_j}X_3 \p_{Y_k} X_3.
\end{gather*}
\end{enumerate}
Therefore, $P(y,D)$ satisfies \eqref{loc-coord-1} if and only if $X_1,X_2$ and $X_3$ are such that 
 \begin{gather*}
 \zed_1(Y,X_1, \nabla_Y X_1)=\zed_2(Y,X_2, \nabla_Y X_2)=\zed_3(Y,X_3, \nabla_Y X_3)=0. 
 \end{gather*}
 Notice that the system is not coupled, and each equation can be solved independently, and therefore this case includes the other two cases of the Proposition.
 
Since $a_{12}(Y)\not=0,$ $a_{13}(Y)\not=0$ and $a_{23}(Y)\not=0$ and $\Theta_j\not=0,$ for $|Y_j|$ small enough, $j=1,2,3,$  the first order PDE for $\zed_1$ is non-characteristic with respect to $\Sigma_2$ or $\Sigma_3.$ Therefore, fixed an initial data 
$X_1(Y_1,0,Y_3)$ (or $X_1(Y_1,Y_2,0)$), there exists a unique  $X_1(Y)$ which is $C^\infty$ in a neighborhood of $q,$ satisfying 
$\zed_1(Y,X_1, \nabla_Y X_1)=0.$

 Similarly,  the differential equation $\zed_2$ is non-characteristic with respect to $\Sigma_1$ or $\Sigma_3$ and the one for $\zed_3,$  is non-characteristic with respect to $\Sigma_1$ or $\Sigma_2.$  Therefore, once suitable initial data is chosen, they have  unique solutions near $q.$
\end{proof}

Now we introduce a variation of the spaces defined by Beals \cite{Beals4}.

\begin{definition}\label{BSDEF}  Let $\Sigma_j\subset \Omega,$ $j=1,2,3,$ be $C^\infty$  closed hypersurfaces intersecting transversally at $\Gamma_{jk}=\Sigma_j \cap \Sigma_k$ and at $\Gamma=\Sigma_1\cap\Sigma_2\cap \Sigma_3.$   Given a point $q\in \Omega,$  fix a neighborhood $U$ of $q$ and fix local coordinates $y$ in $U.$ We say that 
$u\in H_{\loc}^{s, k_1,k_2,k_3}(U, \{y\}),$ $k_1,k_2,k_3\in \mr_+$ and $ s\in \mr,$  if $u$ satisfies the following conditions:
\begin{enumerate}[C1.]
\item\label{away} $u\in H_{\loc}^{s+k_1+k_2+k_3}(U),$ provided $U\subset  \Omega \setminus (\Sigma_1 \cup \Sigma_2\cup \Sigma_3).$ 
\item\label{near-one} If $q\in \Sigma_1\setminus (\Sigma_2\cup \Sigma_3),$ and $y=(y_1,y'')$ $ y''=(y_2,\ldots, y_n),$ are local coordinates such that $\Sigma_1=\{y_1=0\},$ and $P(y,D)$ satisfies \eqref{norm-form-1h}, $u\in H_{\loc}^{s, k_1,k_2,k_3}(U, \{y\})$   if it satisfies
\begin{gather}
 \lan D_{y''}\ran^{k_2+k_3} \vphi  u \in H^{s+k_1}(U), \;\ \vphi\in \Coi(U) \label{cond-OS}
\end{gather}
and similarly for $q\in \Sigma_2$ or $\Sigma_3.$
\item \label{near-two} If $q\in (\Sigma_1\cap \Sigma_2)\setminus \Sigma_3,$ and  $y=(y_1,y_2,y'')$  $y''=(y_3,\ldots, y_n)$ are local coordinates such that $\Sigma_j=\{y_j=0\},$  $j=1,2,$ and $P(y,D)$ satisfies \eqref{norm-form-2h}, $u\in H_{\loc}^{s, k_1,k_2,k_3}(U,\{y\})$ if it satisfies
\begin{gather}
\begin{gathered}
\lan D_{y_1}, D_{y''}\ran^{k_1} \lan D_{y_2}, D_{y''}\ran^{k_2}  \lan D_{y''}\ran^{k_3} \vphi u\in H^s(U),\; \vphi\in \Coi(U).
\end{gathered} \label{cond-TS}
\end{gather}
Similarly for $q\in \Sigma_1\cap \Sigma_3$ or $q\in \Sigma_2\cap \Sigma_3.$
\item\label{near-three} If $q\in \Sigma_1\cap \Sigma_2\cap \Sigma_3$ and in neighborhood $U$ of $q$ and in local coordinates $y=(y_1,y_2,y_3,y''),$  $y''=(y_4,\ldots, y_n),$
 such that  $\Sigma_j=\{y_j=0\},$ $j=1,2,3$ and  $P(y,D)$ satisfies \eqref{loc-coord-1}, $u\in H_{\loc}^{s, k_1,k_2,k_3}(U,\{y\})$ if it satisfies
\begin{gather} 
\begin{gathered} 
\lan D_{y_1}, D_{y''} \ran ^{k_1} \lan D_{y_2}, D_{y''} \ran ^{k_2} \lan D_{y_3}, D_{y''} \ran ^{k_3} 
\vphi u \in H^s(U), \;\ \vphi\in \Coi(U).
\end{gathered}\label{defBSP}
\end{gather}
\end{enumerate}
 \end{definition}
 
 The main difficulty with working with these spaces is that they depend on the choice of the coordinates that satisfy   \eqref{cond-OS}, \eqref{cond-TS} or \eqref{defBSP}.    One should think that the correct way of defining the spaces $H_{\loc}^{s,k_1,k_2,k_3}(\Omega)$ in such that they are invariant,  have properties {\bf P.\ref{inclusion}} to
 {\bf P.\ref{CIA}} in Proposition \ref{BSP2}  and satisfy Proposition \ref{BSP3} below would be to define them as the family of $u\in H_{\loc}^s(\Omega)$ such that
 \begin{gather}
 \mcw_{23}^{k_1} \mcw_{13}^{k_2}\mcw_{12}^{k_3} \ \vphi u \in H^s(\Omega), \vphi\in \Coi(\Omega),\label{VF-BSP}
 \end{gather}
 where $\mcw_{jk}$ are the Lie algebras of vector fields defined in \eqref{123Q}.  These spaces are, in principle, smaller than the spaces defined above. This works perfectly for $k_j\in \mn,$  and there would be no need to put the operator in normal form  in the definition or to prove Proposition \ref{BSP3} below. However,  working with this invariant formulation and $k_j\in \mn,$ one would only be able to prove Theorem \ref{triple-0} for symbols in $S^{m-}$ instead of $S^m,$  and one would also need to assume that $-m-\ha\in \mn.$ This would be less than desirable. To prove Theorem \ref{triple-0}  
  one does need the spaces for $k_j\in \mr_+$ and there is no obvious way of extending the spaces \eqref{VF-BSP} to include $k_j\in \mr_+,$ and keeping the properties {\bf P.\ref{prod-clos}} and {\bf P.\ref{CIA}}. 
 
 We will show that indeed distributions which are conormal to one hypersurface and are in $H_{\loc}^{s,k_1,k_2,k_3}(U,\{y\}),$ $k_j\in \mn,$  for some choice of $\{y\},$ also satisfy  \eqref{VF-BSP} in $U.$ This in particular shows that $I^\mu(\Omega,\Sigma_j)\cap H_{\loc}^{s,k_1,k_2,k_3}(U,\{y\}),$ $k_j\in \mn,$ does not depend of the choice of $\{y\}.$  This  observation will allow us to  circumvent this difficulty  for solutions to a semilinear wave equation \eqref{Weq} and which are conormal to $\Sigma_j,$ $j=1,2,3$ for $t<-1.$  We show that the family of elements of $H_{\loc}^{s,k_1,k_2,k_3}(U,\{y\})$ which satisfy \eqref{Weq}, with $v$ conormal to $\Sigma_1\cup\Sigma_2\cup \Sigma_3$ is independent of the choice of $\{y\}.$
 
 It is implicit in our proofs that enough conormal regularity of the solution for $\{t<-1\}$ propagates to $\Omega$ which implies that, at least  for solutions of semilinear wave equations and for $k_j\in \mn,$ these definitions of $H_{\loc}^{s,k_1,k_2,k_3}(U'\{y\})$ coincide.

First we need to establish properties of composition of functions with elements of 
$H_{\loc}^{s,k_1,k_2,k_3}(U,\{y\})$ and we collect them in the following 
\begin{prop}\label{BSP2} Let $\Sigma_j\subset \Omega,$ $j=1,2,3,$ be $C^\infty$  closed hypersurfaces intersecting transversally at $\Gamma_{jk}=\Sigma_j \cap \Sigma_k$ and at $\Gamma=\Sigma_1\cap\Sigma_2\cap \Sigma_3.$   Given a point $q\in \Omega$  and a neighborhood $U$ of $q,$ and coordinates $y$ valid in $U$ which satisfy one of the assumptions of Definition \ref{BSDEF}, the corresponding spaces $H^{s, k_1,k_2,k_3}_{\loc}(U,\{y\})$ satisfy the following properties:
\begin{enumerate}[{\bf P.1}]
\item\label{inclusion} $H_{\loc}^{s,k_1,k_2,k_3}(U,\{y\}) \subset H_{\loc}^{s',k_1',k_2',k_3'}(U,\{y\}),$ provided $s\geq s',$  $k_j\geq k_j',$ $j=1,2,3.$    In view of that we define
\begin{gather*}
H_{\loc}^{s-,k_1,k_2,k_3}(U,\{y\})=\bigcap_{\eps>0} H_{\loc}^{s-\eps,k_1,k_2,k_3}(U,\{y\}), \\
H_{\loc}^{s,\infty,k_2,k_3}(U,\{y\}) \bigcap_{k_1\in \mr_+} H_{\loc}^{s,k_1,k_2,k_3}(U,\{y\}), \text{ and similarly for other indices. }
\end{gather*}
 
\item\label{Dist} If $a_j>0$ and $a_1+a_2+a_3=1,$ then
\begin{gather}
\begin{gathered}
H_{\loc}^{s+1, k_1,k_2,k_3}(U,\{y\}) \subset H_{\loc}^{s, k_1+a_1,k_2+a_2,k_3+a_3}(U,\{y\}).
\end{gathered}\label{inclusion-0}
\end{gather}
\item \label{inc-LI} If $s\geq 0$ and $k_j>\frac{n}{6},$ then 
\begin{gather}
H_{\loc}^{s-, k_1,k_2,k_3 }(U,\{y\}) \subset L_{\loc}^{\infty}(U). \label{lin-inc}
\end{gather}
\item\label{prod-clos} If $k_j>\frac{n}{6},$ then 
$H_{\loc}^{0-, k_1,k_2,k_3 }(U,\{y\})$  is closed under multiplication, and for $\del>0$ small enough
\begin{gather}
\begin{gathered}
||\vphi u \vphi v||_{-\del,k_1,k_2,k_3} \leq C ||\vphi u||_{-\del,k_1,k_2, k_3} ||\vphi v||_{-\del,k_1,k_2, k_3}.
\end{gathered}\label{norm-prod}
\end{gather}
\item \label{CIA} If  $k_j>\frac{n}{6}+1,$ $j=1,2,3,$ $s\in \mn,$ and  $f(y,z)\in C^\infty(U\times \mr),$
\begin{gather*}
\text{ if } u \in H_{\loc}^{s-,k_1,k_2,k_3}(U,\{y\}), \text{ then }  f(y,u)\in H_{\loc}^{s-,k_1,k_2,k_3}(U,\{y\}), 
\end{gather*}
 Moreover, if $u$ is supported in a compact subset $K\Subset U,$ and if 
 \begin{gather*}
  ||u|||_{-\del,k_1,k_2,k_3}\leq C_\del, 
 \end{gather*}
for $\del$  small enough, and $\vphi\in C_0^\infty,$ then there exists a constant $\tilde C,$ depending on $\del,$ $f,$ $\vphi$ such that 
\begin{gather}
||\vphi f(y,u )|||_{-\del,k_1,k_2,k_3}\leq \tilde C.\label{norm-comp}
\end{gather}
\end{enumerate}
\end{prop}

We will prove this Proposition in an appendix at the end of the paper.  The proofs of items {\bf P.\ref{inclusion}} and {\bf P.\ref{Dist}} easily follow from the definition.   The proof of {\bf P.\ref{inc-LI}} and {\bf P.\ref{prod-clos}}  when $n=3,$ are left as exercise for the reader in \cite{Beals4}. The proof for $n>3$  is more technical and we include it here for completeness. Property {\bf P.\ref{CIA}} is not stated in \cite{Beals4}, even in the three dimensional case.  Property {\bf P.\ref{CIA}} is enough for our purposes, but it is not sharp, $s\geq 0$ and $k_j>\frac{n}{6},$ $j=1,2,3,$ should be enough.

Next we establish mapping properties for the fundamental solution of $P(y,D)$ acting on the spaces $H_{\loc}^{s,k_1,k_2,k_3}(U,\{y\}).$ The case $n=3$ was proved in \cite{SaWang1}. 
\begin{prop}\label{BSP3}     Let $\Sigma_j\subset \Omega,$ $j=1,2,3,$ be $C^\infty$  closed hypersurfaces intersecting transversally at $\Gamma_{jk}=\Sigma_j \cap \Sigma_k$ and at $\Gamma=\Sigma_1\cap\Sigma_2\cap \Sigma_3.$   Let $q\in \Omega,$  let $U$ be  
a neighborhood of $q$ and let $y$ be coordinates in $U$ which satisfy the assumptions of Definition \ref{BSDEF}. for $s\in \mr$ and $k_j\in \mr_+,$  let   $H^{s, k_1,k_2,k_3}_{\loc}(U,\{y\})$ be the corresponding spaces defined in either \eqref{cond-OS}, \eqref{cond-TS} or \eqref{defBSP}.  Suppose that $U$ is bicharacteristically convex with respect to $P(y,D).$  If  $E_+$ denotes the forward fundamental solution to $P(y,D)$ and $\vphi, \psi \in C_0^\infty(U),$  then
\begin{gather}
\psi E_+ \vphi : H^{s, k_1,k_2,k_3}(U,\{y\})\longrightarrow H^{s+1, k_1,k_2,k_3}(U,\{y\}),  \label{mappinge+}
\end{gather}
is a bounded linear operator in term of the norms given by \eqref{defBSP}.
\end{prop}
\begin{proof}  We analyze the case where $q\in \Sigma_1\cap\Sigma_2\cap \Sigma_3.$ The proofs of the  other cases are very similar. We start by proving this result for $k_j\in \mn,$ $j=1,2,3.$ This  is based on a commutator method due to Bony \cite{Bony3,Bony4}, see also Melrose and Ritter \cite{MelRit}.  In this case, the operator $P(y,D)$ is given by \eqref{loc-coord-1} and then
\begin{gather*}
[P(y,D), \p_{y_1}] =  (\p_{y_1} b_{12}) \p_{y_1}\p_{y_2}+ (\p_{y_1} b_{13}) \p_{y_1}\p_{y_3}+ (\p_{y_1} b_{23}) \p_{y_2}\p_{y_3}+
\sum_{j=1, k=4}^n  (\p_{y_1} b_{jk}(y))  \p_{y_j}\p_{y_k}+\mcl^*,
\end{gather*}
where $\mcl^*$ is a differential operator of order one.   
Since $P(y,D)$ is strictly hyperbolic, $b_{23}\not=0,$ and using the formula for $P(y,D)$ given by \eqref{loc-coord-1} we can write
\begin{gather*}
(\p_{y_1} b_{23}) \p_{y_2}\p_{y_3}= \frac{\p_{y_1} b_{23}}{b_{23}}(P-  b_{12} \p_{y_1}\p_{y_2}- b_{13} \p_{y_1}\p_{y_3}-
\sum_{j=1, k=4}^n b_{jk}(y)  \p_{y_j}\p_{y_k}-\mcl),
\end{gather*}
and therefore we obtain
\begin{gather}
[P(y,D), \p_{y_1}]= a_1(y) P(y,D)+ \mcl_{11}(y,D) \p_{y_1}+ \sum_{k=4}^n \mcl_{1k} \p_{y_k}+ \wt\mcl_1. \label{comm-1}
\end{gather}
We can argue in the same way to obtain
\begin{gather}
[P(y,D), \p_{y_m}]= a_m(y) P(y,D)+ \mcl_{m1}(y,D) \p_{y_1}+ \sum_{k=4}^n \mcl_{mk} \p_{y_k}+\mcl_m , \; m=4, \ldots n.\label{comm-m}
\end{gather}
 Therefore,  if $P(y,D)u=f,$  
$\mcu_1^T=( u, \p_{y_1} u, \p_{y_4} u, \ldots, \p_{y_n} u)$ and $\mcf_1^T=(f, \p_{y_1} f, \ldots, \p_{y_n} f),$  one gets a system 
\begin{gather*}
\mcp_1 \mcu_1= \mcm_1 \mcf_1 \in H_{\loc}^{s}(\Omega),
\end{gather*}
where $\mcp_1$ is a matrix of operators with diagonal principal part $P(y,D) \Id,$ $\mcm_1$ is a matrix of $C^\infty$ functions. Since $\mcp$ is a strictly hyperbolic system, we conclude that $\mcu_1  \in H^{s+1}(\Omega).$ 

Using the fact that for any three operators (as long as the compositions are well defined),
\begin{gather}
[A, BC]= B[A,C] + [A,B]C, \label{comm-rel}
\end{gather}
we deduce from  \eqref{comm-1} and \eqref{comm-m} that, for $m\geq 4,$
\begin{gather*}
[P(y,D), \p_{y_1}\p_{y_m}]= A_{1m}(y) P +  B_{1m}(y)\p_{y_1} P + C_{1m}(y)\p_{y_m} P + \mcl_1(y,D) \p_{y_1}+ \mcl_m(y,D) \p_{y_m} + \\
\mcl_{11}(y,D) \p_{y_1}^2+\mcl_{1m}(y,D) \p_{y_1}\p_{y_m}+ \sum_{k=4}^n \mcl_{1k} \p_{y_1} \p_{y_k} + \sum_{k=4}^n \mcl_{1k} \p_{y_m} \p_{y_k}.
\end{gather*}
We can rewrite this as
\begin{gather*}
[P(y,D), \p_{y_1}\p_{y_m}]= \sum_{|\beta|=0}^1 a_{\beta}(y)(\p_{y_1},\p_{y_m})^\beta P(y,D)+ \sum_{|\beta|=0}^2 \mcl_{\beta}(y,D) (\p_{y_1},\p_{y''})^{\beta}, 
\end{gather*}
where $a_\beta\in \CI$ and $\mcl_{\beta}$  are differential operators with $\CI$ coefficients of order one.

Using \eqref{comm-rel} and induction we arrive at
\begin{gather*}
[P(y,D), (\p_{y_1}, \p_{y''})^{\alpha}]=  \sum_{|\beta|=0}^{|\alpha|-1} a_{\beta}(y) (\p_{y_1}, \p_{y''})^\beta P(y,D)+ 
\sum_{|\beta|=0}^{|\alpha|}  \mcl_{\beta} (\p_{y_1}, \p_{y''})^\beta.
\end{gather*}

Using the same argument, we have
\begin{gather*}
[P(y,D), (\p_{y_1}, \p_{y''})^{\alpha_1} (\p_{y_2}, \p_{y''})^{\alpha_2} (\p_{y_3}, \p_{y''})^{\alpha_3} ]=  \\
\sum_{j=1}^3 \sum_{|\beta_j|=0}^{|\alpha_j|-1} a_{\beta_1,\beta_2,\beta_3}(y) (\p_{y_1}, \p_{y''})^{\beta_1} (\p_{y_2}, \p_{y''})^{\beta_2} (\p_{y_3}, \p_{y''})^{\beta_3} P(y,D)+ \\
\sum_{j=1}^3 \sum_{|\beta_j|=0}^{|\alpha_j|} \mcl_{\beta_1,\beta_2,\beta_3}(y,D) (\p_{y_1}, \p_{y''})^{\beta_1} (\p_{y_2}, \p_{y''})^{\beta_2} (\p_{y_3}, \p_{y''})^{\beta_3}.
\end{gather*}

If 
\begin{gather*}
\mcu^T=( u, (\p_{y_1},\p_{y''})^{\beta_1}(\p_{y_2},\p_{y''})^{\beta_2}(\p_{y_3},\p_{y''})^{\beta_3}u) \text{ and }  
\mcf^T=( f, (\p_{y_1},\p_{y''})^{\beta_1}(\p_{y_2},\p_{y''})^{\beta_2}(\p_{y_3},\p_{y''})^{\beta_3}f),
\end{gather*}  
we obtain a system

\begin{gather*}
\mcp \mcu= \mcf,
\end{gather*}
where $\mcp$ is a square matrix of differential operators of order two, and its principal part is $P(y,D) \Id.$ So the system is strictly hyperbolic, and conclude that $\mcu \in H_{\loc}^{s+1}(\Omega).$

This proves the result for $k_j\in \mn,$ $j=1,2,3.$  The case $k_j\in \mr_+$ follows from the Stein-Weiss interpolation Theorem \cite{SteWei}, and we refer the reader to  the proof of Proposition 2.13 of  \cite{SaWang1} for more details.
\end{proof}

Now we consider the invariance of $H_{\loc}^{s,k_1,k_2,k_3}(U,\{y\})$ under change of variables.   If $U\cap \Sigma_j=\emptyset,$ $j=1,2,3,$ then 
it is well known that $H_{\loc}^{s,k_1,k_2,k_3}(U,\{y\})=H_{\loc}^{s+k_1+k_2+k_3}(U)$ is independent of the choice of coordinates.  The other cases are more subtle and we introduce the following spaces:
\begin{definition}\label{CImch}  If $U\subset \Omega$ is an open subset, we say that $u\in \mch_{\loc}^{s,k_1,k_2,k_3}(U)$ if  
 given any two sets of local coordinates $y$ and $Y$ which are defined in $U$ and satisfy the same one of the conditions of Definition \ref{BSDEF},  and $Y=\Psi^{-1}(y),$ $u\in H_{\loc}^{s,k_1,k_2,k_3}(U,\{y\})$ if and only if $\Psi^* u\in H_{\loc}^{s,k_1,k_2,k_3}(U,\{Y\}).$

 We say that $u\in \mch_{\loc}^{s,k_1,k_2,k_2}(\Omega)$ if for every open subset $U\subset \Omega$  equipped with local coordinates as in Definition \ref{BSDEF}, $\vphi u \in \mch^{s,k_1,k_2,k_3}(U)$ for all $\vphi\in C_0^\infty(U).$
\end{definition}

\begin{prop}\label{one-hyp-in} Let $\Sigma_1\subset \Omega$ be as above.  Then for any $U\subset \Omega$ which has  local coordinates valid in $U$ such that the condition C.\ref{near-one} of Definition \ref{BSDEF}  holds, 
$u\in I^{\mu}(U,\Sigma_1)\cap H_{\loc}^{s,k_1,k_2,k_3}(U,\{y\})=\mch_{\loc}^{s,k_1,k_2,k_3}(U).$  In other words, the space $u\in I^{\mu}(U,\Sigma_1)\cap H_{\loc}^{s,k_1,k_2,k_3}(U,\{y\})$ is invariant under change of variables that fix $\Sigma_1.$
\end{prop}
\begin{proof}   Let $\{y\}$ be local coordinate sin $U$ such that $\Sigma_1=\{y_1=0\}.$ A change of coordinates $Y=\Psi(y)$ in $U$ that fixes $\Sigma_1,$ must satisfy
\begin{gather}
\begin{gathered}
Y_1=y_1 X_1(y), \;\ |X_1(y)|>0, \;\ Y''= Y''(y), \;\ j=2,3,\ldots,n,  \text{ or }\\
y_1=Y_1 Z_1(y), \;\ |Z_1(Y)|>0, \;\ y''= W''(Y), \;\ j=2,3,\ldots,n,
\end{gathered}\label{eq-y1}
\end{gather}
and hence
\begin{gather}
\begin{gathered}
\p_{y_1}= (X_1+y_1\p_{y_1}X_1) \p_{Y_1}+ \p_{y_1} Y_2 \p_{Y_2}+ \p_{y_1} Y_3 \p_{Y_3}, \\
\p_{y_j}=(\frac{1}{X_1}\p_{y_j}X_1 ) Y_1\p_{Y_1}+ \sum_{k=2}^n \p_{y_j} Y_k \p_{Y_k}, \;\ j=2,3,\ldots, n.
\end{gathered}\label{PBpy}
\end{gather}
According to \eqref{cond-OS}, if $k_j\in \mn,$ $j=1,2,3,$ $u\in H_{\loc}^{s,k_1,k_2,k_3}(\Omega),$ if
\begin{gather}
\lan \p_{y_1}, \p_{y''}\ran^{k_1} \lan \p_{y''}\ran^{k_1+k_2} \vphi u \in H^{s}, \;\ \vphi\in\Coi(U). \label{COND-1}
\end{gather}

Notice that if $u \in I^{\mu}(\Omega,\Sigma_1),$ then $u \in C^\infty(\Omega\setminus \Sigma_1),$ so it is in $\mch_{\loc}^{s,k_1,k_2,k_3}(U)$ for any open subset $U\subset \Omega \setminus \Sigma_1$ and all $s, k_j\in \mr,$ $j=1,2,3,$ and in this case the choice of coordinates is irrelevant.  So we may assume that $U$ is small enough so 
that $u$ is given by an oscillatory integral

\begin{gather*}
\vphi u(y)= \int_\mr e^{i y_1 \eta_1} a(y'', \eta_1) d\eta_1, \;\ a_1\in S^{\mu+\frac{n}{4}-\ha}(\mr^{n-1}\times \mr),
\end{gather*}
with $a(y'',\eta_1)$ compactly supported in $y'',$ but then  \eqref{COND-1} is equivalent to  that
\begin{gather*}
\lan D_{y_1}\ran^{k_1} u \in H^{s}(U)  \text{ or }    a(y'',\eta_1) (1+|\eta_1|)^{k_1+s} \in L^2(\mr^n), \; \alpha \in \mn^{n-1}.
\end{gather*}

But of course, the pull-back of $\vphi u$  by $\Psi$ is a conormal distribution to $\Sigma_1$ and using \eqref{eq-y1} it would also be given by an oscillatory integral 
\begin{gather*}
\Psi^*(\vphi u)(Y)= \int_\mr e^{i Y_1 Z_1(Y) \eta_1} a(W''(Y), \eta_1) \ d\eta_1, \;\ 
\end{gather*}
If one sets $Z_1(y)\eta_1=\xi,$ then as in the proof of Theorem 18.2.9 of \cite{HormanderV3}, 
\begin{gather*}
\Psi^* u(Y)= \int_\mr e^{i Y_1 \xi} a(W''(Y), (Z_1(Y))^{-1}\xi) (Z_1(y))^{-1} \ d\xi= 
\int_\mr e^{i Y_1 \xi} b(Y'',\xi)d\xi+ \mce, \\
\text{ where } \mce \in V^\infty \text{ and } \\ 
b(Y'',\xi)\sim \sum_{j=0}^\infty \frac{1}{j!} \lan D_{Y''}, D_\xi\ran^j a(W''(0,Y''), (Z_1(0,Y''))^{-1}\xi) (Z_1(0,Y''))^{-1},
\end{gather*}
and since there exists a constant $C$ such that $\frac{1}{C} \leq |Z_1(0,y'')|\leq C,$ it follows that
\begin{gather*}
a(y'',\eta_1) (1+|\eta|)^{s+k_1} \in L^2(\mr^n) \text{ if and only if }    b(Y'',\xi) (1+|\xi|)^{s+k_1} \in L^2(\mr^n).
 \end{gather*}
 This ends the proof of Proposition \ref{one-hyp-in}.
 \end{proof}

We will need the following result about the inclusion  of distributions conormal  to a hypersurface into  
$\mch_{\loc}^{s,k_1,k_2,k_3}(\Omega):$
\begin{prop}\label{incBSP} Let  $\Omega\subset \mr^n$ be an open neighborhood of the origin and $y=(y_1,y_2,y_3,y''),$  be coordinates in $\Omega$  such that $\Sigma_j=\{ y_j=0\},$ $j=1,2,3.$   If  $v_j \in I^{\msi}(\Omega, \Sigma_j),$ $j=1,2,3$ and  $s\geq 0,$ then
\begin{gather}
\begin{gathered}
v_j \in H_{\loc}^{s}(\Omega), \; s<-m-\ha,  \; j=1,2,3,\\
 v_1\in \mch_{\loc}^{r,\ka,\infty,\infty}(\Omega), \;\ v_2\in \mch_{\loc}^{r,\infty, \ka, \infty}(\Omega) \text{ and } 
 v_3\in \mch_{\loc}^{r,\infty, \infty,\ka}(\Omega), \\ \text{ provided } \ka\geq 0, \; r\geq 0 \text{  and }  \ka+r<-m-\ha,
 \end{gathered}\label{inBSP1}
 \end{gather} 
\end{prop}
The proof is a very simple application of the Fourier transform and details (for $n=3$)  can be found in the proof of Proposition 2.15 of \cite{SaWang1}.

Next we consider solutions of semilinear wave equations which are in $H_{\loc}^{s,k_1,k_2,k_3}(U,\{y\}).$
\begin{prop}\label{CIBS}  Let $q\in \Omega\cap \{t=c\},$ and let $U\subset U'\subset \Omega$ be relatively compact open subsets. Suppose that $U$ is bicharacteristically convex with respect to $P(y,D).$  Moreover, suppose that  there exist local coordinates $y$ in $U'$  such that one of the normal forms required in Definition \ref{BSDEF} holds in $U'$ and let $H^{s-,k_1,k_2,k_3}(U,\{y\})$ denote the restriction of functions $u\in H_{\loc}^{s-,k_1,k_2,k_3}(U',\{y\})$ to  $U.$

 Suppose that $u\in H_{\loc}^s(\Omega),$ $s>\novt,$ satisfies \eqref{Weq} and   $u\in H^{s-,k_1,k_2,k_3}(U \cap\{t<c\}, \{y\}),$ with $s\geq 0,$ and $k_j>\novt+3,$ $j=1,2,3,$  then $u\in H_{\loc}^{s-,k_1,k_2,k_3}(U,\{y\}).$
\end{prop}
\begin{proof}  Let $\chi \in \CI(\mr),$ be such that $\chi(t)=1,$ $t>c-\del$ and $\chi=0$ for $t<c-2\del,$ then
\begin{gather*}
P(y,D) (\chi u)= \chi f(y,u) + [P(y,D),\chi] u.
\end{gather*}
Notice that $[P(y,D),\chi] u\in H_{\loc}^{s-1-, k_1,k_2,k_3}(U)$ is supported in $t\in [c-2\del,c-\del].$  If 
\begin{gather*}
P(y,D) v= [P(y,D),\chi] u, \\
v=0, \;\ t<c-2\del,
\end{gather*}
it follows from Proposition \ref{BSP3} that $v \in H^{s-, k_1,k_2,k_3}(U).$  If one writes $u=\chi u+(1-\chi) u$ and $\chi u=v+w,$ then
\begin{gather*}
P(y,D) w= \chi f(y, w+ \vartheta), \;\ \vartheta= v+(1-\chi) u \in H^{s-, k_1,k_2,k_3}(U, \{y\}), \\
w=0, \;\ t<c-2\del.
\end{gather*}
Suppose that $k_1\leq \min\{k_2,k_3\}.$  Since at every point, $\gamma\in T^*\Omega,$ $\eta_j,$ is elliptic for some 
$j\in \{1,2,3\},$  it follows that $u\in H^{s+k_1-}(U\cap\{t<-1\},\{y\}).$ Then, since $k_1>\novt,$ standard propagation of Sobolev regularity for nonlinear wave equation shows that $u\in H^{s+k_1-}(U).$   But in view of Property \ref{Dist} in Proposition \ref{BSP2},
\begin{gather*}
w\in H^{s+k_1-}(U)\subset H^{s-, \frac{k_1}{3}, \frac{k_1}{3}, \frac{k_1}{3}}(U,\{y\}).
\end{gather*}
But $\frac{k_1}{3}>\frac{n}{6}+1,$ so in view of Property \ref{CIA}  in Proposition \ref{BSP2}, 
$\chi f(y, w+ \vartheta) \in H^{s-, \frac{k_1}{3}, \frac{k_1}{3}, \frac{k_1}{3}}(U,\{y\})$ and then by Proposition \ref{BSP3},
\begin{gather*}
w \in H^{s+1-, \frac{k_1}{3}, \frac{k_1}{3}, \frac{k_1}{3}}(U,\{y\})  \subset  H^{s-, \frac{k_1+1}{3}, \frac{k_1+1}{3}, \frac{k_1+1}{3}}(U,\{y\}).
\end{gather*}
We repeat this argument and conclude that $w \in  H^{s-, \frac{k_1+2}{3}, \frac{k_1+2}{3}, \frac{k_1+2}{3}}(U,\{y\}),$ and after finitely many steps, we conclude that $w \in  H^{s-, k_1,k_2,k_3}(U,\{y\}).$ This proves the Proposition.
\end{proof}

The following is a consequence of Proposition \ref{CIBS}:
\begin{prop}\label{CINL}   Let $q\in \{t=c\}\cap \Omega$ and let $U\subset \Omega$ be a relatively compact bicharacteristically convex neighborhood of $q.$    
  Suppose that  $u\in H^s(U),$ $s>\novt,$ satisfies 
\eqref{Weq} and  $u \in \mch^{s-,k_1,k_2,k_3}(U\cap \{t<c\}),$ $s\geq 0$ and $k_j>\novt+3,$ then
$u \in \mch^{s-,k_1,k_2,k_3}(U).$
\end{prop}
\begin{proof}  Suppose $y$ and $Y$ are local coordinates in  $U$  satisfying the same one of the conditions of Definition \ref{BSDEF} and let $Y=\Psi^{-1}(y).$  Suppose that $u \in H^{s-,k_1,k_2,k_3}(U,\{y\}),$ then in particular 
$u \in H^{s-,k_1,k_2,k_3}(U\cap \{t<c\},\{y\}).$ But by assumption, $\Psi^* u \in H^{s-,k_1,k_2,k_3}(U\cap \{\Psi^*t<c\},\{Y\}),$ and so by 
Proposition \ref{CIBS}, $\Psi^* u \in H^{s-,k_1,k_2,k_3}(U,\{Y\}).$   The same argument shows that  if 
$\Psi^* u \in H^{s-,k_1,k_2,k_3}(U,\{Y\}),$ then  $u \in H^{s-,k_1,k_2,k_3}(U,\{y\}).$
\end{proof}

As a consequence of Proposition \ref{CINL} and Proposition \ref{one-hyp-in} we have that 
\begin{prop} \label{NL-C-INV} Let $u\in H_{\loc}^s(\Omega),$ $s>\novt,$ be a solution to \eqref{Weq}.  Let $\mu\geq 0$ and $k_j\geq \novt+3,$ and suppose that for $t<-1,$ $u\in \mch_{\loc}^{\mu-,k_1,k_2,k_3}(\Omega\cap\{t<-1\})$ and that for every $q\in \{t=-1\}$ there exists a neighborhood $U$ of $q$ such that $u\in \mch^{\mu-,k_1,k_2,k_3}(U\cap\{t<-1\}).$ Then 
$u\in \mch_{\loc}^{\mu-,k_1,k_2,k_3}(\Omega).$
\end{prop} 
\begin{proof} Let $U$ be a neighborhood of $q\in \{t=-1\}$ such that $u\in \mch_{\loc}^{\mu,k_1,k_2,k_3}(U\cap\{t<-1\}).$ 
 By shrinking $U$ if necessary, we may assume it is bicharacteristically convex.  Then Proposition \ref{CINL} shows that $u\in \mch_{\loc}^{\mu,k_1,k_2,k_3}(U).$  Since the support of $\mcy$ is compact, this shows that the result is true in $\Omega\cap \{t<-1+\eps\},$ for some $\eps.$ We repeat the argument for $q\in \{t=-1+\eps\}$ instead of $\{t=-1\}.$ Since the support of $\mcy(y)$ is compact, we prove the result by repeating this argument finitely many times.
\end{proof}
The following result will be important in the proof of Theorem \ref{triple-0}:
\begin{prop}\label{aux-pr-t0} Let $u\in H_{\loc}^s(\Omega),$ $s>\novt,$ be a solution to \eqref{Weq} and suppose that the initial data
$v=v_1+v_2+v_3,$ $v_j\in I^{m-\novf+\ha}(\Omega,\Sigma_j),$ $m<-\ha(n+7).$ Then $u\in \mch_{\loc}^{0-,-m-\ha,-m-\ha,-m-\ha}(\Omega).$
\end{prop}
\begin{proof}  Since $v_j\in I^{\msi}(\Omega,\Sigma_j),$ it follows from Proposition \ref{one-hyp-in}, 
\begin{gather*}
v_1 \in \mch_{\loc}^{0-, -m-\ha, \infty,\infty}(\Omega), \;\ v_2 \in \mch_{\loc}^{0-, \infty, -m-\ha, \infty}(\Omega) \text{ and }  v_3 \in \mch_{\loc}^{0-, \infty, \infty, -m-\ha}(\Omega),
\end{gather*}
and therefore for every $q\in \{t=-1\}$ there exists a relatively compact neighborhood $U$ of $q$ such that $u\in \mch^{0-,-m-\ha,-m-\ha,-m-\ha}(U\cap\{t<-1\}).$  Since $m<-\ha(n+7),$ $m-\ha>\novt+3,$ and therefore the result follows from Proposition \ref{NL-C-INV}.
\end{proof}

We recall some results from \cite{SaWang1} about products of conormal distributions. 

\begin{prop}\label{prod-three-dist}  Let  $\Omega\subset \mr^n$ be an open neighborhood of the origin and $y=(y_1,y_2,y_3)$  be coordinates in $\Omega$  such that $\Sigma_j=\{y_j=0\},$ $j=1,2,3,$  and let  
$v_j \in I^{m-\frac{n}{4}+\ha}(\Omega, \Sigma_j),$ $j=1,2,3,$ $m<-1.$  If 
\begin{gather*}
a_1(\eta_1,y_2,y_3,y''), a_2(\eta_2,y_1, y_3,y''), a_3(\eta_3,y_1,y_2,y'')  \in  S^m(\mr \times \mr^{n-1}), \\
\text{  are respectively the principal symbols of } v_1, v_2 \text{ and } v_3
\end{gather*}
 then
\begin{gather}
\begin{gathered}
v_1 v_2 = w_{12} +  \mce_{12}, \ v_1 v_3 = w_{13} +  \mce_{13}, \text{ and }
v_2 v_3 = w_{23} +  \mce_{23},
\end{gathered}\label{prod-2}
\end{gather}
 where
 \begin{gather}
 \begin{gathered}
 w_{12}(y)= \int_{\mr} e^{i( \eta_1 y_1+ y_2\eta_2)} a_1(\eta_1,0, y_3,y'') a_2( \eta_2,0, y_3,y'')  \ d\eta_1d\eta_2, \\ \mce_{12} \in  \mch_{\loc}^{0-, -m+\ha,-m-\ha,\infty}(\Omega)+ \mch_{\loc}^{0-, -m-\ha,  -m+\ha ,\infty}(\Omega), \\
 w_{13}(y)= \int_{\mr} e^{i( \eta_1 y_1+ y_3\eta_3)} a_1(\eta_1,y_2, 0,y'') a_3( \eta_3,0, y_2,y'')  \ d\eta_1d\eta_3, \\ \mce_{13} \in  \mch_{\loc}^{0-, -m+\ha,\infty,-m-\ha}(\Omega)+ \mch_{\loc}^{0-, -m-\ha,\infty,-m+\ha}(\Omega), \\
 w_{23}(y)= \int_{\mr} e^{i( \eta_2 y_2+ y_3\eta_3)} a_2(\eta_2,y_1,0,y'') a_3(\eta_3, y_1,0,y'')  \ d\eta_2d\eta_3, \\ \mce_{23} \in  \mch_{\loc}^{0-, \infty,-m+\ha,-m-\ha}(\Omega)+ \mch_{\loc}^{0-, \infty,-m-\ha,-m+\ha}(\Omega).
 \end{gathered}\label{prod-21}
 \end{gather}
 
 The product of three distributions satisfies
 \begin{gather}
\begin{gathered}
v_1v_2v_3 = V+ \mce, \text{ where } \\
V(y)= \int_{\mr^3} e^{i(y_1\eta_1+y_2\eta_2+y_3\eta_3)} a_1(\eta_1,0, 0,y'') a_2( \eta_2,0, 0,y'')  a_3(\eta_3,0,0,y'') \ d\eta_1d\eta_2d\eta_3,  \\
\mce \in  \mch_{\loc}^{0, -m+\ha,-m-\ha,-m-\ha}(\Omega) + \mch_{\loc}^{0, -m-\ha,-m+\ha,-m-\ha}(\Omega) + \mch_{\loc}^{0, -m-\ha,-m-\ha,-m+\ha}(\Omega)
\end{gathered}\label{prod-3}
\end{gather}
\end{prop}

One can interpret this result in terms of distributions which are conormal to $\Gamma,$ but with product-type symbols.  Suppose $v_j,$ $j=1,2,3$ is conormal to $\Sigma_j,$ and has principal symbol $a_j\in S^m(N^* \Sigma_j, \Omega^\ha_{N^*\Sigma_j}).$ Suppose that  in local coordinates in which $\Sigma_j=\{y_j=0\},$  $a_j$ is of the form
$a_j(y_j',\eta_j) |dy_j'|^\ha |d\eta_j|^\ha$ $y_1'=(y_2,\ldots, y_n),$ $y_2'=(y_1,y_3,\ldots, y_n)$ and $y_3'=(y_1,y_2,y_4,\ldots, y_n).$  
We rephrase \eqref{prod-3} as
\begin{gather*}
v_1v_1v_3= V+ \mce, \text{ where } V \text{ is a conormal distribution to } \Gamma \text{ with a product-type principal symbol } \\
\sigma(V)= (a_1 a_2a_3)|_{\Gamma},
\end{gather*}
which in local coordinates $y=(y_1,y_2,y_3,y''),$ $y''=(y_4,\ldots, y_n),$  has the form
\begin{gather*}
\sigma(V)(y'',\eta_1,\eta_2,\eta_3)=a_1(0,0,0,y'',\eta_1)a_2(0,0,0,y'',\eta_2)a_3(0,0,0,y'',\eta_3)  |d y''|^\ha |d \eta_1 d \eta_2 d \eta_3|^\ha.
\end{gather*}

Notice that any transformation that fixes the hypersurfaces $\Sigma_j=\{y_j=0\},$ $j=1,2,3,$ is of the form
\begin{gather*}
\widetilde y= \Psi(y_1,y_2,y_3,y'')= (y_1 F_1(y), y_2 F_2(y), y_3 F_3(y), F_4(y), \ldots, F_n(y)),
\end{gather*}
and therefore, $d \widetilde{y_1}=(F_1(y)+ y_1\p_{y_1} F_1(y)) d y_1,$ and
\begin{gather*}
\p_{y_1}= (F_1(y)+ y_1 \p_{y_1} F_1(y)) \p_{\widetilde{y_1}}+ y_2 \p_{y_1} F_2(y) \p_{\widetilde{y_2}}+y_3 \p_{y_1} F_3(y) \p_{\widetilde{y_3}}
\end{gather*}
Therefore,
\begin{gather*}
(\Psi^* a_1 a_2 a_3)|_{\Gamma}= a_1(0,0, 0,\widetilde{y''},\widetilde{\eta_1}) a_2(0,0, 0,\widetilde{y''},\widetilde{\eta_2}) a_3(0,0, 0,\widetilde{y''},\widetilde{\eta_3}) |d \widetilde{y''}|^\ha|d\widetilde{\eta_1}d\widetilde{\eta_2}d\widetilde{\eta_3}|^\ha|,
\end{gather*}
and we conclude that  the restriction $a_j|_{\Gamma}$ is well defined.

\section{ The Proof of Theorem \ref{triple-0}}\label{gen-sing}

We follow the strategy of the proof of Theorem 4.1 of \cite{SaWang1}, which is in part based on the strategy of the proof of theorem 4.1 of \cite{Beals4}. The main novelty here is that we allow arbitrary $\CI$ nonlinearities $f(y,u),$ while in \cite{SaWang1} the nonlinear term $f(y,u)$ is supposed to be a polynomial in $u$ with $C^\infty$ coefficients.

First we discuss the global behavior of the solution $u$ of \eqref{Weq} and later we will analyze the microlocal behavior of $u$ near $\Gamma$ in the directions where $\lan \eta_j \ran \geqs \lan \eta \ran,$ $j=1,2,3,$ and show that new singularities will exist on $\mcq.$ 

 Recall that the initial data of \eqref{Weq} is of the form  $v=v_1+v_2+v_3,$ with $v_j \in I^{m-\novt+\ha}(\Omega,\Sigma_j)$ and 
 $m<-\ha(n+7).$    So we conclude from Proposition \ref{aux-pr-t0} that $u\in \mch_{\loc}^{0-,-m-\ha, -m-\ha, -m-\ha}(\Omega).$  
 
 We recall the formula for the expansion of the triple product $v_1v_2v_3$ given by \eqref{prod-3}, and motivated by this, we define the following spaces:
\begin{gather*}
\mck^{0,m}(\Omega)= \mch_{\loc}^{0-, -m+\ha,-m-\ha,-m-\ha}(\Omega)+ \mch_{\loc}^{0-, -m-\ha,-m+\ha,-m-\ha}(\Omega)+ \mch_{\loc}^{0-, -m-\ha,-m-\ha,-m+\ha}(\Omega), \\
\mck^{1,m}(\Omega)= \mch_{\loc}^{1-, -m+\ha,-m-\ha,-m-\ha}(\Omega)+ \mch_{\loc}^{1-, -m-\ha,-m+\ha,-m-\ha}(\Omega)+ \mch_{\loc}^{1-, -m-\ha,-m-\ha,-m+\ha}(\Omega). 
\end{gather*}
To simplify the notation, we will define, 
\begin{gather}
F(y,u(y))= \mcy(y) f(y,u(y)), \label{Not-f}
\end{gather}
where $f(y,u)$ and $\mcy(y)$ are as in Theorem \ref{triple-0}, so $F(y,u(y))$ is compactly supported and its support is contained  in $ \Omega\cap \{t\geq-1\}.$

We use the forward fundamental solution to $P(y,D),$ which we have denoted by $E_+,$ and the fact that $P(y,D)v=0,$  to write
\begin{gather}
u=v+E_+(F(y,u)),\label{u-EP}
\end{gather}
and  we use  Proposition \ref{decomp-01}  to write
\begin{gather}
\begin{gathered}
v_j=\nu_j+\mce_j, \;  \nu_j\in \ido{}^{m-\novf+\ha}(\Omega,\Sigma_j),  \; \mce_j \in \CI, \text{ and hence } 
v= \nu+ \mce, \\ \text{ where }  \mce \in C^\infty \text{ and } \nu= \nu_1+\nu_2+\nu_3, \;\ \nu_j \in \ido{}^{m-\novf+\ha}(\Omega,\Sigma_j).
\end{gathered}\label{reg-vj}
\end{gather}
To simplify the notation, we use \eqref{reg-vj} to write
\begin{gather}
u= v+E_+(F(y,u(y))= \nu+\mcw, \text{ where } \mcw= \mce + E_+( F(y,u)). \label{defmcw-0}
\end{gather}
Later we will need  the fact that since $u=\mce+\nu$ and since $\nu(0,y'')=0,$
\begin{gather}
\mcw(0,y'')=u(0,y''). \label{wof0}
\end{gather}

Since $m<-\ha(n+7),$ it follows that $-m+\ha>\novt+3,$ and we deduce from Property C.\ref{CIA} of Proposition \ref{BSP2} that 
$F(y,u)\in  \mch_{\loc}^{0-,-m-\ha, -m-\ha, -m-\ha}(\Omega).$   Then, it follows from Proposition \ref{BSP3} that
\begin{gather}
E_+( F(y,u)) \in \mch_{\loc}^{1-,-m-\ha, -m-\ha, -m-\ha}(\Omega), \label{NecReg}
\end{gather}
and so we conclude that
\begin{gather}
\mcw= \mce + E_+( F(y,u)) \in \mch_{\loc}^{1-,-m-\ha, -m-\ha, -m-\ha}(\Omega). \label{defmcw}
\end{gather}

We iterate \eqref{defmcw-0} and obtain
\begin{gather}
u = \nu+ E_+(  F(y, \nu+\mcw)) +\mce. \label{New-Regu-u-N1}
\end{gather}

Our first result separates the terms with higher order of regularity of $F(y,\nu+\mcw):$
\begin{prop}\label{Reg-Terms} Let $u,$ $F(y,u)$ be as above and let $V$ be defined in \eqref{prod-3}. Then
\begin{gather}
\begin{gathered}
F(y, \nu+\mcw)- (\p_u^{3} F)(y,\mcw) V \in \mck^{0,m}(\Omega).
\end{gathered}\label{Reg-terms1}
\end{gather}
\end{prop}
\begin{proof}  We begin by taking the Taylor expansion of order three in $u$ of $F(y,u)$ centered at $\mcw:$ 
\begin{gather}
\begin{gathered}
F(y,\mcw+\nu)= \sum_{j=0}^3 \frac{1}{j!} (\p_u^j F)(y,\mcw) \nu^j +
\frac{1}{3!}\nu^4 \int_0^1 (\p_u^4F)(y, \mcw+t\nu) (1-t)^3 dt.
\end{gathered}\label{Taylor1}
\end{gather}
 First we consider the third order Taylor polynomial
\begin{gather*}
T_3(y)=\sum_{j=0}^3 \frac{1}{j!} (\p_u^{j} F)(y,\mcw) \nu^j.
\end{gather*}

We know that
\begin{gather*}
\nu^2= \nu_1^2+ \nu_2^2+ \nu_2^2+ 2 \nu_1\nu_2+2 \nu_1\nu_3+2 \nu_2\nu_3, \\
\nu^3 =  \nu_1^3+\nu_2^3+\nu_3^3+ 3\nu_1^2 \nu_2+3\nu_1^2 \nu_3+ 3\nu_2^2 \nu_1+3\nu_2^2 \nu_3+3\nu_3^2 \nu_1+3\nu_3^2 \nu_2+
6 \nu_1\nu_2\nu_3.
\end{gather*}
Therefore we write
\begin{gather*}
T_3(y)=F(y,\mcw)+ \Theta(y) + (\p_u^3 f)(y,\mcw)\nu_1\nu_2\nu_3, \\
\Theta(y)=    (\p_u F)(y,\mcw)(\nu_1+\nu_2+\nu_3) + \ha(\p_u^2 f)(y,\mcw)(\nu_1^2+ \nu_2^2+ \nu_2^2+ 2( \nu_1\nu_2+ \nu_1\nu_3+ \nu_2\nu_3)) + \\
\frac{1}{3!} (\p_u^3 f)(y,\mcw)( \nu_1^3+\nu_2^3+\nu_3^3+ 3\nu_1^2 \nu_2+3\nu_1^2 \nu_3+ 3\nu_2^2 \nu_1+3\nu_2^2 \nu_3+3\nu_3^2 \nu_1+3\nu_3^2 \nu_2).
\end{gather*}
First we consider the terms $\nu_j,$ $\nu_i^{2},$ $\nu_j^3$ or $\nu_j\nu_k$ or $\nu_j^2\nu_k.$  It follows from Proposition \ref{BSP2} that for $\alpha_j\geq 1,$ $\alpha_j \in \mn,$
\begin{gather*}
\nu_1^{\alpha_1}  \in \mch_{\loc}^{0-, -m-\ha, \infty, \infty}(\Omega), \;\ \nu_2^{\alpha_2}  \in \mch_{\loc}^{0-, \infty, -m-\ha, \infty}(\Omega), \\
 \nu_3^{\alpha_3}  \in \mch_{\loc}^{0-, \infty, \infty, -m-\ha}(\Omega), \;\ 
\nu_1^{\alpha_1} \nu_2^{\alpha_2} \in \mch_{\loc}^{0-, -m-\ha, -m-\ha, \infty}(\Omega), \\ 
\nu_1^{\alpha_1} \nu_3^{\alpha_3} \in \mch_{\loc}^{0-, -m-\ha, \infty, -m-\ha}(\Omega), \;\ 
\nu_2^{\alpha_2} \nu_3^{\alpha_3} \in \mch_{\loc}^{0-, \infty, -m-\ha, -m-\ha}(\Omega).
\end{gather*}
In fact,  in view of Proposition \ref{prod1}, the terms with  $\alpha_j>1$ are smoother,  but we do not need this now. We know from \eqref{defmcw}  that $\mcw \in  \mch_{\loc}^{1-, -m-\ha, -m-\ha, -m-\ha}(\Omega),$ and since $m<-\novt-\sha,$ it follows from Property {\bf P.\ref{CIA}} of Proposition \ref{BSP2} that for any $j\in \mn,$
\begin{gather}
(\p_u^j F)(y,\mcw) \in  \mch_{\loc}^{1-, -m-\ha, -m-\ha, -m-\ha}(\Omega).\label{reg-PW}
\end{gather}
 and therefore from Property {\bf P.\ref{prod-clos}} of Proposition \ref{BSP2} we conclude that for any $j,$ and $\alpha_j\geq 1,$
 \begin{gather}
 \begin{gathered}
\nu_1^{\alpha_1}  (\p_u^j F)(y,\mcw) \in  \mch_{\loc}^{0-, -m-\ha, -m-\ha, -m+\ha}(\Omega), \;\ 
\nu_2^{\alpha_2}  (\p_u^j F)(y,\mcw) \in  \mch_{\loc}^{0-, -m+\ha, -m-\ha, -m-\ha}(\Omega), \\
\nu_3^{\alpha_3}  (\p_u^j F)(y,\mcw) \in  \mch_{\loc}^{0-, -m+\ha, -m-\ha, -m-\ha}(\Omega), \;\
\nu_2^{\alpha_2} \nu_3^{\alpha_3} (\p_u^j F)(y,\mcw)\in  \mch_{\loc}^{0-, -m+\frac12, -m-\ha, -m-\ha}(\Omega), \\
\nu_1^{\alpha_1} \nu_3^{\alpha_3} (\p_u^j F)(y,\mcw) \in  \mch_{\loc}^{0-, -m-\frac12, -m+\frac12, -m-\ha}(\Omega), \;\
\nu_1^{\alpha_1} \nu_2^{\alpha_2} (\p_u^j F)(y,\mcw) \in  \mch_{\loc}^{0-, -m-\frac12, -m-\ha, -m+\frac12}(\Omega),
\end{gathered}\label{singleP}
\end{gather}
So we conclude that
\begin{gather*}
T_3(y) - (\p_u^3 F)(y,\mcw)\nu_1\nu_2\nu_3 \in \mck^{0,m}(\Omega).
\end{gather*}
But, we know from Proposition \ref{prod-three-dist} that if $V$ is given by \eqref{prod-3}, then
$\nu_1\nu_2\nu_3 -V\in \mck^{0,m}(\Omega),$ and again from \eqref{reg-PW} and Property {\bf P.\ref{prod-clos}} of Proposition \ref{BSP2}  we conclude that
\begin{gather}
\begin{gathered}
T_3(y) - (\p_u^3 F)(y,u) V \in \mck^{0,m}(\Omega).
\end{gathered}\label{terms}
\end{gather}

Next we consider the fourth order remainder in the Taylor's expansion:
\begin{lemma}\label{TREM}  Under the hypotheses of Theorem \ref{triple1}, for any $t\in [0,1],$  we have
\begin{gather}
\begin{gathered}
\nu^4 (\p_u^4 F)(y,\mcw+t\nu)= \mcb_1(t)+ \mcb_2(t)+\mcb_3(t), \\
\mcb_1(t)\in \mch_{\loc}^{0-,-m+\ha,-m-\ha,-m-\ha}(\Omega),   \mcb_2(t)\in \mch_{\loc}^{0-,-m-\ha,-m+\ha,-m-\ha}(\Omega), \text{ and } \\ \mcb_3(t)\in  
 \mch_{\loc}^{0-,-m-\ha,-m-\ha,-m+\ha}(\Omega) \\
 \text{ and moreover for every } \del>0 \text{ and } \vphi \in C_0^\infty(\Omega),  \text{ there exists }  \\ \text{ a constant } C>0 \text{ such that for all } t\in [0,1], \\
 ||\vphi \mcb_1(t)||_{-\del,-m+\ha,-m-\ha,-m-\ha} \leq C\\
 ||\vphi \mcb_2(t)||_{-\del,-m-\ha,-m+\ha,-m-\ha} \leq C\\
 ||\vphi \mcb_3(t)||_{-\del,-m-\ha,-m-\ha,-m+\ha} \leq C.
 \end{gathered}\label{TREM1}
\end{gather}
\end{lemma}
\begin{proof}  We begin by expanding the term $\nu^4=(\nu_1+\nu_2+\nu_3)^4:$
\begin{gather*}
\nu^4= \nu_1^4 + \nu_2^4+\nu_3^4+ 4\nu_1^3 \nu_2+ 4\nu_1^3 \nu_3+ 4\nu_2^3 \nu_1+ 4\nu_2^3 \nu_3+  \\
4\nu_3^3 \nu_1+4\nu_3^3 \nu_2+ 12 \nu_1^2 \nu_2\nu_3+ 12 \nu_2^2 \nu_1\nu_3+ 12 \nu_3^2 \nu_1\nu_2+
6 \nu_1^2 \nu_2^2+ 6 \nu_1^2\nu_3^2+ 6\nu_2^2\nu_3^2.
\end{gather*}
We see that
\begin{gather}
\begin{gathered}
\p_{y_1} \left[ \nu_1^4   (\p_u^4 F)(y,\mcw+t\nu)\right]= (\p_{y_1}\nu_1^4) (\p_u^4 F)(y,\mcw+t\nu)+ \\
\nu_1^4 (\p_{y_1}\p_u^4 F)(y,\mcw+t\nu)+\nu_1^4 (\p_u^5 F)(y,\mcw+t\nu)(t\p_{y_1} \nu+ \p_{y_1}\mcw)=\\
 (\p_{y_1}\nu_1^4) (\p_u^4 F)(y,\mcw+t\nu)+\frac{t}{5} (\p_u^5 F)(y,\mcw+t\nu)\p_{y_1} \nu_1^5 \\
\nu_1^4 (\p_{y_1}\p_u^4 F)(y,\mcw+t\nu)+\nu_1^4 (\p_u^5 F)(y,\mcw+t\nu)(t\p_{y_1}( \nu_2+\nu_3)+ \p_{y_1}\mcw)
\end{gathered}\label{TP1}
\end{gather}

Now we appeal to Proposition \ref{prod1} to conclude that
\begin{gather}
\nu_1^j\in I^{m-(j-1)k(m) -\novf+\ha}(\Omega,\Sigma_1) \subset \mch_{\loc}^{0-, -m+(j-1)k(m)-\ha, \infty,\infty}(\Omega).
\end{gather}
Since $m<-\novt-\sha,$ $k(m)\geq 1$  and hence
\begin{gather}
\begin{gathered}
\p_{y_1} \nu_1^4 \in \mch_{\loc}^{0-, -m+3k(m)-\tha, \infty,\infty}(\Omega)\subset  \mch_{\loc}^{0-, -m+\tha, \infty,\infty}(\Omega), \\
\p_{y_1} \nu_1^3 \in \mch_{\loc}^{0-, -m+2k(m)-\tha, \infty,\infty}(\Omega)\subset  \mch_{\loc}^{0-, -m+\ha, \infty,\infty}(\Omega), \\
\p_{y_1} \nu_1^2 \in \mch_{\loc}^{0-, -m+k(m)-\tha, \infty,\infty}(\Omega)\subset  \mch_{\loc}^{0-, -m-\ha, \infty,\infty}(\Omega).
\end{gathered}\label{regder-0}
\end{gather}
But we also know that
\begin{gather}
\begin{gathered}
\p_{y_1} \nu_2 \in \mch_{\loc}^{0-, \infty, -m-\ha,\infty}(\Omega), \\
\p_{y_1} \nu_3 \in \mch_{\loc}^{0-, \infty, \infty, -m-\ha,}(\Omega).
\end{gathered}\label{TP11}
\end{gather}
Since $\nu\in \mch_{\loc}^{0-,-m-\ha,-m-\ha,-m-\ha}(\Omega)$ and  $\mcw \in \mch_{\loc}^{1-,-m-\ha,-m-\ha,-m-\ha}(\Omega),$ and $m<-\novt-\sha,$ we know from  Proposition \ref{CIA} that
\begin{gather*}
(\p_{y_1}\p_u^4 F)(y,\mcw+t\nu),  \; (\p_u^5 F)(y,\mcw+t\nu), \;  (\p_u^4 F)(y,\mcw+t\nu) \text{ and } \\ (\p_u^5 F)(y,\mcw+t\nu)
\in \mch_{\loc}^{0-,-m-\ha,-m-\ha,-m-\ha}(\Omega)
\end{gather*}
We also know that $\p_{y_1}\mcw \in  \mch_{\loc}^{0-,-m-\ha,-m-\ha,-m-\ha}(\Omega)$ and so we conclude that
\begin{gather}
\begin{gathered}
\p_{y_1} \left[\nu_1^4   (\p_u^4 F)(y,\mcw+t\nu)\right]  \in \mch_{\loc}^{0-,-m-\ha,-m-\ha,-m-\ha}(\Omega), \text{ and hence }\\
\nu_1^4   (\p_u^4 F)(y,\mcw+t\nu) \in \mch_{\loc}^{0-,-m+\ha,-m-\ha,-m-\ha}(\Omega). \
\end{gathered}\label{TA1}
\end{gather}
The same argument used with respect to $y_2$ and $y_3$ respectively shows that
\begin{gather}
\begin{gathered}
\nu_2^4   (\p_u^4 F)(y,\mcw+t\nu) \in \mch_{\loc}^{0-,-m-\ha,-m+\ha,-m-\ha}(\Omega) \text{ and }\\
\nu_3^4   (\p_u^4 F)(y,\mcw+t\nu)\in  \mch_{\loc}^{0-,-m-\ha, -m-\ha,-m+\ha}(\Omega).
\end{gathered} \label{TA2}
\end{gather}
Now we consider the terms with $\nu_1^3.$ We again write
\begin{gather*}
\p_{y_1}\left[\nu_1^3(\nu_2+\nu_3) (\p_u^4 F)(y,\mcw+t\nu)\right]=
(\p_{y_1}\nu_1^3)(\nu_2+\nu_3) (\p_u^4 F)(y,\mcw+t\nu)+ \\
\nu_1^3(\p_{y_1}(\nu_2+\nu_3) ) (\p_u^4 F)(y,\mcw+t\nu)+ 
\nu_1^3(\nu_2+\nu_3) (\p_{y_1}\p_u^4 F)(y,\mcw+t\nu)+  \\
\nu_1^3(\nu_2+\nu_3) (\p_u^5 F)(y,\mcw+t\nu)( t\p_{y_1}\nu+ \p_{y_1}\mcw)=\\
(\p_{y_1}\nu_1^3)(\nu_2+\nu_3) (\p_u^4 F)(y,\mcw+t\nu)+\frac{t}{4}(\nu_2+\nu_3) (\p_u^5 F)(y,\mcw+t\nu)( \p_{y_1}\nu_1^4)+ \\
\nu_1^3(\p_{y_1}(\nu_2+\nu_3) ) (\p_u^4 F)(y,\mcw+t\nu)+ 
\nu_1^3(\nu_2+\nu_3) (\p_{y_1}\p_u^4 F)(y,\mcw+t\nu)+  \\
\nu_1^3(\nu_2+\nu_3) (\p_u^5 F)(y,\mcw+t\nu)( t\p_{y_1}(\nu_2+\nu_3)+ \p_{y_1}\mcw).
\end{gather*}
Using \eqref{regder-0} we conclude that
\begin{gather}
\begin{gathered}
\p_{y_1}\left[\nu_1^3(\nu_2+\nu_3) (\p_u^4 F)(y,\mcw+t\nu)\right] \in \mch_{\loc}^{0-,-m-\ha,-m-\ha,-m-\ha}(\Omega) \text{ and therefore } \\
\nu_1^3(\nu_2+\nu_3) (\p_u^4 F)(y,\mcw+t\nu) \in \mch_{\loc}^{0-,-m+\ha,-m-\ha,-m-\ha}(\Omega).
\end{gathered}\label{TA3}
\end{gather}
Following this argument with respect to $y_2$ and $y_3$ we also find that
\begin{gather}
\begin{gathered}
\nu_2^3(\nu_1+\nu_3) (\p_u^4 F)(y,\mcw+t\nu) \in  \mch_{\loc}^{0-,-m-\ha,-m+\ha,-m-\ha}(\Omega) \text{ and }\\
\nu_3^3(\nu_1+\nu_2)   (\p_u^4 F)(y,\mcw+t\nu) \in  \mch_{\loc}^{0-,-m-\ha,-m-\ha,-m+\ha}(\Omega).
\end{gathered} \label{TA4}
\end{gather}
The terms in $\nu_j^2\nu_k^2,$ $j\not=k,$  and $\nu_j^2\nu_k\nu_m,$  $j\not=k,$ $j\not=m$ and $k\not=m$ can be handled in the same way and we obtain
\begin{gather*}
\begin{gathered}
\p_{y_1}\left[ \nu_1^2 \nu_2^2 (\p_u^4 F)(y,\mcw+t\nu)\right] \in \mch_{\loc}^{0-,-m-\ha,-m-\ha,-m-\ha}(\Omega), \\
\p_{y_1}\left[ \nu_1^2 \nu_3^2 (\p_u^4 F)(y,\mcw+t\nu)\right] \in \mch_{\loc}^{0-,-m-\ha,-m-\ha,-m-\ha}(\Omega), \\
\p_{y_2}\left[ \nu_2^2 \nu_3^2 (\p_u^4 F)(y,\mcw+t\nu)\right] \in \mch_{\loc}^{0-,-m-\ha,-m-\ha,-m-\ha}(\Omega), \\
\p_{y_1}\left[ \nu_1^2 \nu_2 \nu_3 (\p_u^4 F)(y,\mcw+t\nu)\right] \in \mch_{\loc}^{0-,-m-\ha,-m-\ha,-m-\ha}(\Omega), \\
\p_{y_2}\left[ \nu_2^2 \nu_1\nu_3  (\p_u^4 F)(y,\mcw+t\nu)\right] \in \mch_{\loc}^{0-,-m-\ha,-m-\ha,-m-\ha}(\Omega), \\
\p_{y_3}\left[ \nu_3^2 \nu_1\nu_2  (\p_u^4 F)(y,\mcw+t\nu)\right] \in \mch_{\loc}^{0-,-m-\ha,-m-\ha,-m-\ha}(\Omega).
\end{gathered}
\end{gather*}
The estimates in \eqref{TREM1}  follow by applying \eqref{norm-comp} and \eqref{norm-prod} at each step of the proof. This ends the proof of  Lemma \ref{TREM}.
\end{proof}

To finish the proof of the proof of Proposition \ref{Reg-Terms} one needs to show that
\begin{gather*}
\int_0^1 \nu^4 (\p_u^4 F)(y, \mcw+t\nu)(1-t)^3 dt \in \mck^{m,0}(\Omega).
 \end{gather*}
 The integral is a well defined Riemann integral, as all functions here are continuous, the only issue is the boundedness of the integral in these spaces, but this follows from the estimates in \eqref{TREM1} for each $\mcb_j(t),$ $j=1,2,3.$ 
 \end{proof}

Next we consider the regularity of the solution $u$ near $\Gamma.$  Let $\chi_j\in C_0^\infty(U_j),$ $U_j$ a small enough neighborhood of $q_j\in \Gamma,$ $1\leq j \leq N,$ and suppose that $\sum_{j=1}^N \chi_j=1$ in a neighborhood of $\Gamma$ in the support of $\mcy(y).$
It follows from Proposition \ref{prod-three-dist} that
\begin{gather*}
(1-\chi)V \in \mck^{0,m}(\Omega),
\end{gather*}
We already know that $\mcw\in \mch_{\loc}^{1-,-m-\ha,-m-\ha-m-\ha}(\Omega),$ and therefore, since $m<-\ha(\novt+7),$
$(\p_u^3F)(y,\mcw)\in  \mch_{\loc}^{1-,-m-\ha,-m-\ha-m-\ha}(\Omega).$  Since 
$\mch_{\loc}^{1-,-m-\ha,-m-\ha-m-\ha}(\Omega)$ is contained in each factor of $\mck^{1,m}(\Omega).$
\begin{gather}
(1-\sum_{j=1}^N \chi_j) V (\p_u^3F)(y,\mcw)\in \mck^{0,m}(\Omega), \label{cut-off-away}
\end{gather}
and therefore, putting this together with \eqref{Reg-terms1}, we find that
\begin{gather}
F(y,\mcw+\nu) -\sum_{j=1}^N \chi_j  (\p_u^3F)(y,\mcw) V \in \mck^{0,m}(\Omega). \label{Reg-terms10}
\end{gather}
So we need to discuss the regularity of the terms $\chi_j V (\p_u^3F)(y,\mcw)$  in the region near 
\begin{gather*}
N^*\Gamma\setminus 0=\{ y_1=y_2=y_3=0, \eta''=0\},
\end{gather*}
where $\lan \eta_j\ran\geqs \lan \eta \ran,$ $j=1,2,3,$ so  we define a pseudodifferential operator  $A(D)$ with symbol
\begin{gather}
\sigma(A)(\eta)=\psi(\frac{\eta_1}{\lan \eta \ran})\psi(\frac{\eta_2}{\lan \eta \ran})\psi(\frac{\eta_3}{\lan \eta \ran}),\label{symbol-A}
\end{gather}
where $\psi(s)\in C^\infty(\mr),$ $\psi(s)= 0$ for $|s|\leq\ha$ and $\psi(s)=1$ for $|s|\geq 1.$

We prove the following:
\begin{prop}\label{triple1}    Let  $\Sigma_j,$ $j=1,2,3$ and $\mcq$ satisfy the hypotheses of Theorem \ref{triple}.  Let $q\in \Gamma,$  let $U$ be a neighborhood of $q$ and let  $y=(y',y''),$ $y'=(y_1,y_2,y_3)$ be local coordinates in $U$  such that $\Sigma_j=\{y_j=0\},$ $j=1,2,3$ and $P(y,D)$ is given by \eqref{loc-coord-1}. Let $\chi\in C_0^\infty(U).$  Suppose that  $v=v_1+v_2+v_3,$  $v_j\in I^{m-\novf+\ha}(\Omega,\Sigma_j),$  $m<-\ha(n+7),$ $j=1,2,3.$  Let   
\begin{gather*}
a_1(y_2,y_3,y'',\eta_1), \; a_2(y_1,y_3,y''\eta_2), \; a_3(y_1,y_2,y'',\eta_3)\in S^m(\mr^{n-1}\times \mr), 
\end{gather*}  
 be the principal symbols of  $v_1,$ $v_2$ and $v_3$ respectively, and assume they are elliptic.
    Let $V$ be the principal part of the product $v_1v_2v_3$ given by \eqref{prod-3}. 
 Let  $u$ be the solution to \eqref{Weq} and let $\mcw$ be defined in \eqref{defmcw}.    If $A\in \Psi^0(\Omega)$ is defined in \eqref{symbol-A}, then
\begin{gather}
\begin{gathered}
A\chi \left[ \left[(\p_u^3 F)(y,\mcw) -(\p_u^3 F)(0,0,0,y'', u(0,0,0,y''))\right]V\right]  \in H^{-3m-\ha+\frac{r}{2}},\text{ provided }   r\in (0, 1-\frac{2}{-m-\ha}), \\
\text{ while } A\chi \left((\p_u^3 F)(0,0,0,y'', u(0,0,0,y''))V\right) \in H^{-3m-\ha-}\setminus H^{-3m-\ha}.
\end{gathered}\label{reguL}
\end{gather}
\end{prop}
\begin{proof}   This is a consequence of the following Lemma, which was proved in Proposition 4.3 of \cite{SaWang1}:  

\begin{lemma}\label{asyFT}  Let $\alpha(\eta)=\alpha_1(\eta_1) \alpha_2(\eta_2) \alpha_3(\eta_3),$ with $\alpha_j(\eta_j)\in S^{m}(\mr)$ and $m<-\fha.$  Let  $b(\eta_1,\eta_2,\eta_3)$ be such that for all $\del>0,$
 \begin{gather}
 \lan \eta_1 \ran^{-m-\ha} \lan \eta_2 \ran^{-m-\ha} \lan \eta_3 \ran^{-m-\ha} \lan (\eta_1,\eta_2,\eta_3)\ran ^{1-\del}b(\eta_1,\eta_2,\eta_3) \in L^2(\mr^3).\label{estb}
  \end{gather} 
 Then in the conic region
 \begin{gather*}
  \begin{gathered}
  \Upsilon_{\mu_0,\mu_1} =\Upsilon_{\mu_0,\mu_1}(\eta_1) \cup \Upsilon_{\mu_0,\mu_1}(\eta_2) \cup \Upsilon_{\mu_0,\mu_1}(\eta_3), \\
 \Upsilon_{\mu_0,\mu_1}(\eta_1)=\left\{  \; \mu_0<  |\eta_2/\eta_1|<\mu_1 , \; \mu_0 < |\eta_3/\eta_1|<\mu_1\right\},\\
\Upsilon_{\mu_0,\mu_1}(\eta_2)= \{ \mu_0 < |\frac{\eta_1}{\eta_2}|<\mu_1, \;  \mu_0 < |\frac{\eta_3}{\eta_2}|<\mu_1\}, \;\
\Upsilon_{\mu_0,\mu_1}(\eta_3)=\{ \mu_0 < |\frac{\eta_1}{\eta_3}|<\mu_1, \;  \mu_0 < |\frac{\eta_2}{\eta_3}|<\mu_1\}.
  \end{gathered} 
    \end{gather*}
  the convolution $\alpha\star b$ satisfies
 \begin{gather}
 \begin{gathered}
 \alpha\star b(\eta_1,\eta_2,\eta_3)= \alpha(\eta_1,\eta_2,\eta_3) \int_{\mr^3} b(\eta) d\eta + \mce(\eta_1,\eta_2,\eta_3), \text{ where } \\
   \int_{\Upsilon_{\mu_0,\mu_1}}\left| \lan(\eta_1,\eta_2,\eta_3)\ran^{-3m-\tha+r/2} \mce(\eta_1,\eta-2,\eta_3)\right|^2 d\eta_1d\eta_2d\eta_3<\infty, \text{ for } r+\frac{2}{-m-\ha} < 1.
  \end{gathered}\label{asyFT1}
    \end{gather}
 By symmetry the same result holds in the regions where either $\eta_2$ or $\eta_3$ are elliptic. 
\end{lemma}

Now we can finish the proof of Proposition \ref{triple1}. We define $b(\eta_1,\eta_2,\eta_3,y'')$ to be the partial Fourier transform in $y'=(y_1,y_2,y_3)$ of $\chi(y) (\p_u^3 F)(y,\mcw):$
\begin{gather*}
b(\eta_1,\eta_2,\eta_3, y'')= \mcf_{y'}(\chi(y) (\p_u^3 F)(y,\mcw))(\eta_1,\eta_2,\eta_3,y'')=\\
\int_{\mr^3} e^{i(y_1\eta_1+y_2\eta_2+y_3\eta_3)} \chi(y_1,y_2,y_3,y'') (\p_u^3 F)(y_1,y_2,y_3,y'',\mcw(y_1,y_2,y_3,y'')) d\eta_1 d\eta_2 d\eta_3.
\end{gather*}
We know from \eqref{prod-3} that
\begin{gather*}
\mcf_{y'}(V)= a(\eta_1,\eta_2,\eta_3,y'')=a_1(\eta_1,0,0,y'')a_2(\eta_2,0,0,y'')a_3(\eta_3,0,0,y''),
\end{gather*}
and we also know that $\chi(y)(\p_u^3 F)(y,\mcw))\in \mch^{1-,-m-\ha,-m-\ha,-m-\ha}(U)$ and in particular,
\begin{gather*}
\lan \eta_1\ran^{-m-\ha} \lan \eta_2\ran^{-m-\ha} \lan \eta_3\ran^{-m-\ha} \lan (\eta_1,\eta_2,\eta_3) \ran^{1-\eps} b(\eta_1,\eta_2,\eta_3,y'') \in L^2(\mr_{\eta_1,\eta_2,\eta_3}^3\times \mr_{y''}^{n-3}),
\end{gather*}
and so we deduce from Lemma \ref{asyFT} that
\begin{gather*}
\mcf_{y'}( \chi(y) (\p_u^3 F)(y,\mcw) V)(\eta_1,\eta_2,\eta_3,y'')=\frac{1}{(2\pi)^3} a\star b(\eta',y'')=\\ a(\eta_1,\eta_2,\eta_3,y'') \frac{1}{(2\pi)^3} \int_{\mr^3} b(\zeta_1,\zeta_2,\zeta_3,y'') d\zeta_1d\zeta_2d\zeta_3 +  \mce(\eta_1,\eta_2,\eta_3,y'')= \\ a(\eta_1,\eta_2,\eta_3,y'') (\p_u^{3} F)(0,0,0,y'',\mcw(0,0,0,y'')) V(y) + \mce(\eta_1,\eta_2,\eta_3,y''),
\end{gather*}
where $\mce(\eta_1,\eta_2,\eta_2,y'')$ satisfies 
\begin{gather}
\begin{gathered}
\int_{\mr^{n-2}}   \int_{\Upsilon_{\mu_0,\mu_1}}\left| \lan (\eta_1,\eta_2,\eta_3)\ran^{-3m-\tha+r/2} \mce(\eta_1,\eta_2,\eta_3,y'')\right|^2 d\eta_1d\eta_2 d\eta_2 dy'' <\infty.
\end{gathered}\label{asyFT4}
\end{gather}  
Let $A(D)$ be the pseudodifferential operator with symbol $\sigma(A)(\eta)$ given by \eqref{symbol-A}.  We conclude from \eqref{asyFT4} that
\begin{gather*}
A(D) \mce(y)= \mcf^{-1}(\sigma(A)(\eta) \mcf_{y''}(\mce)(\eta_1,\eta_2,\eta_3,\eta''))\in H^{-3m-\tha+r/2}, \text{ provided } r<1-\frac{2}{-m-\ha},
\end{gather*}
where $\mcf_{y''}$ is the partial Fourier transform in $y''$ and $\mcf$ is the Fourier transform in $y=(y',y'').$   Therefore, we conclude from \eqref{asyFT4} that
\begin{gather}
\begin{gathered}
 A(D)\left[ \chi(y) (\p_u^3 F)(y,\mcw) V- \chi(y)(\p_u^{3} F)(0,0,0,y'',\mcw(0,0,0,y'')) V(y)\right]= \\ A(D)\mce(y) \in H^{-3m-\tha+\frac{r}{2}}, \;\ \text{ provided } r<1-\frac{2}{-m-\ha}.
 \end{gathered} \label{AofD}
  \end{gather}
This concludes the proof of Proposition \ref{triple1}
\end{proof}

Now we can finish the proof of Theorem \ref{triple-0}. 
\begin{proof}

 We know from Proposition \ref{Reg-Terms}  and from \eqref{cut-off-away} that
\begin{gather}
F(y, \mcw+\nu)-\sum_{j=1}^N \chi_j(y) (\p_u^{3} F)(y,\mcw)
 \in \mck^{0,m}(\Omega), \label{approx-1}
\end{gather}
If $A(D)$ is a pseudodifferential operator with symbol given by \eqref{symbol-A}, it follows from the definition of $\mck^{0,m},$ 
that
\begin{gather*}
A(D)\left(F(y, \mcw+\nu)-\sum_{j=1}^N \chi_j(y) (\p_u^{3} F)(y,\mcw)\right)
 \in H^{3m-\ha}(\Omega). \label{approx-12}
\end{gather*}

But then in view of \eqref{AofD},
\begin{gather*}
A(D)\left(F(y, \mcw+\nu)-\sum_{j=1}^N \chi_j(y) (\p_u^{3} F)(0,0,0,y'',\mcw(0,0,0,y'')\right)
 \in  H^{-3m-\tha+\frac{r}{2}}, \text{ if } r<1-\frac{2}{-m-\ha}.
\end{gather*}

But recall that $u=v+E_+(F(y,\nu+\mcw)),$ and that $\nu=0$ on $\Gamma$ and $\mcw=u$ on $\Gamma$ so $\nu+\mcw|_{\Gamma}=u|_{\Gamma},$ and this means that  on Lagrangian submanifold $\La$ defined in \eqref{defla}:
\begin{gather*}
\bigcup_{s>0} \exp( s H_p) \left( N^*(\Gamma \setminus 0)\cap p^{-1}(0) \cap \{\lan \eta_j\ran \geqs \lan \eta\ran, \; j=1,2,3\}\right),
\end{gather*}
we have
\begin{gather*}
u-E_+\left( F(y, u(y))|_{\Gamma} V\right)\in H^{3m-\ha+r/2}(\Omega).
\end{gather*}

But we already know  from the work of Bony \cite{Bony5,Bony6}, Melrose and Ritter \cite{MelRit} and S\'a Barreto \cite{SaB,SaB2} that $u$ is conormal to $\mcq$ away from $\Gamma.$   This actually means that the principal part of $u$ on $\La,$ and away from $\Sigma_j,$ $j=1,2,3$ and $\Gamma$  is given by $E_+( (\p_u^3 F)(y,u) V).$ It remains to find the order of the symbol of $u$ on $N^*\mcq$ away from $\gamma.$

 Since $v_j\in I^{m-\novf+\ha}(\Omega,\Sigma_j),$  its principal symbol is of order $m.$  If $V$ is the distribution  given by \eqref{prod-3}, since $\Gamma$ has codimension three, then 
  in the region where  $\lan\eta_j\ran\geqs \lan \eta\ran,$  $V\in I^{3m-\novf+\tha}(\Omega,\Gamma).$
We apply the results of   Greeleaf and Uhlmann \cite{GreUhl} about paired Lagrangian distributions. They prove that $E_+$ is a paired Lagrangian distribution in  $I^{-\tha, -\ha}(N^*\diag, \La),$ where 
\begin{gather*}
N^*\diag=\{(y,\eta,y',\eta')\in T^*(\Omega\times \Omega): y=y', \eta=-\eta'\}, 
\end{gather*}
and $\La$ is the flow-out of $N^*\diag\cap \{p=0\}$ under $H_p.$ Proposition 2.1 of \cite{GreUhl} shows that away from $\Gamma,$ 
$E_+(V)\in I^{3m-\novf}(\Omega, \mcq).$  This concludes the proof of Theorem \ref{triple-0}.
\end{proof}
\section{Appendix: The Proof of Proposition \ref{BSP2}} 

First we prove Property {\bf P.\ref{inc-LI}}.  
\begin{proof}
Suppose that $\vphi\in C_0^\infty(U)$ and $\vphi u\in H^{0-,k_1,k_2,k_3}(U)$ with $k_j>\frac{n}{6},$ we want to show that
$\vphi u\in L^\infty(U)$   and we just need to show that $\mcf(\vphi u)\in L^1(\mr^n).$  We define
\begin{gather}
W_{\vkap}(\eta)=\lan \eta_1,\eta'' \ran^{k_1}\lan \eta_2,\eta'' \ran^{k_2}\lan \eta_3,\eta''\ran^{k_3}\lan \eta\ran^{-\del}, \label{WKA}
\end{gather}
and so

\begin{gather*}
||\mcf(\vphi u)||_{L^1(\mr^n)}= \int_{\mr^n} |\mcf(\vphi u)(\eta)| d\eta= \int_{\mr^n} [W_{\vkap}(\eta)]^{-1} |W_{\vkap}(\eta)|
|\mcf(\vphi u)(\eta)| d\eta\leq \\
\left[\int_{\mr^n} [W{\vkap}(\eta)]^{-2} d\eta\right]^\ha \left[\int_{\mr^n} |W{\vkap}(\eta)|^2 |\mcf(\vphi u)(\eta)|^2 d\eta\right]^\ha=
||\vphi u||_{H^{-\del,k_1,k_2,k_3}} \left[\int_{\mr^n} [W_{\vkap}(\eta)]^{-2} d\eta\right]^\ha 
\end{gather*}
Notice that, if one sets $t=(1+\rho^2) z,$ then
\begin{gather*}
\int_{\mr} (1+\rho^2+|t|^2)^{k} dt= (1+\rho^2)^{k+1} \int_\mr (1+z^2)^{-m} dz= C(1+\rho^2)^{k+1}.
\end{gather*}
So, by setting $\rho=|\eta''|,$  we obtain
\begin{gather}
\int_{\mr^n} [W{\vkap}(\eta)]^{-2} d\eta= C\int_{\mr^3} (1+\eta_1^2+\rho^2)^{-2k_1}(1+\eta_2^2+\rho^2)^{-2k_2}(1+\eta_3^2+\rho^2)^{-2k_3}
\rho^{n-4} d\rho d\eta_1d \eta_2d\eta_3\leq \\
\leq C\int_{\mr} (1+\rho^2)^{-2(k_1+k_2+k_3)+n-1} d\rho, \label{ineq1}
\end{gather}
which converges, since $n-2(k_1+k_2+k_3)<0.$ So, $||\mcf(\vphi u)||_{L^1(\mr^n)} <\infty$ and hence $\vphi u\in L^\infty(U).$
\end{proof}

Now we prove  Property {\bf P.\ref{prod-clos}}
\begin{proof}

The main ingredient in the arguments used  below is the following Lemma:
\begin{lemma}\label{Rauch}(Rauch and Reed \cite{RauRee01}) Suppose $K(\xi,\eta)=\sum_{j=1}^k K_j(\xi,\eta)$ and
\begin{gather}
\sup_{\xi} \int |K_j(\xi,\eta)|^2 d\eta<\infty \text{ or } \sup_{\eta} \int |K_j(\xi,\eta)|^2 d\xi<\infty. \label{RR-D}
\end{gather}
If $f,g\in L^2(\mr^n)$ and $h(\xi)=\int K(\xi,\eta) f(\xi-\eta)g(\eta) d\eta,$ it follows that $h\in L^2(\mr^n)$ and
\begin{gather}
||h||_{L^2}\leq C ||f||_{L^2} ||g||_{L^2}.\label{Rauch1}
\end{gather}
\end{lemma}

Let $\xi=(\xi_1,\xi_2,\xi_3,\xi'')$ and $\eta=(\eta_1,\eta_2,\eta_3,\eta'').$ For $\varkappa=(k_1,k_2,k_3)$ and $\del>0,$ let
$W_{\vkap}(\eta)$ be defined in \eqref{WKA} and let
\begin{gather*}
K_{\vkap}(\xi,\eta)= \frac{ W_{\vkap}(\xi)}{W_{\vkap}(\xi-\eta)W_{\vkap}(\eta)}.
\end{gather*}
Then,
\begin{gather*}
\mcf\left(W_{\vkap}(D) u v\right)(\xi) =\int_{\mr^4} K_{\vkap} (\xi,\eta) \left(W_{\vkap}(\xi-\eta) \widehat{u}(\xi-\eta)\right) 
\left(W_{\vkap}(\eta)\widehat{v}(\eta)\right) d\eta
\end{gather*}
So, according to Lemma \ref{Rauch} we need to prove that $K_{\vkap}(\xi,\eta)$ can be decomposed as a sum 
$K_{\vkap}(\xi,\eta)=\sum_{j=1}^M K_{\vkap,j}(\xi,\eta),$ with $K_{\vkap,j}$ satisfying \eqref{RR-D}.

For $j=1,2,3,$ we shall denote, 
\begin{gather}
\begin{gathered}
E_{j}=\{(\xi_j, \xi'',\eta_j,\eta''): |(\xi_j,\xi'')-(\eta_j,\eta'')|\leq \ha |(\xi_j,\xi'')| \}, \text{ and  }  \\ 
F_{j}=\{(\xi,\eta): |(\xi_j-\eta_j, \xi''-\eta'')|> \ha |(\xi_j,\xi'')| \}.
\end{gathered}\label{ineq-ej}
\end{gather}
Notice that
\begin{gather}
\ha|(\xi_j,\xi'')|\leq |(\eta_j,\eta'')|\leq \tha |(\xi_j,\xi'')| \text{ on } E_j. \label{ineq-ej1}
\end{gather}
Let $\chi_{{}_{E_{j}}}$ and $\chi_{{}_{F_{j}}}$ denote the characteristic functions of $E_{j}$ and $F_{j}$ respectively.  For 
$J=(j_{1},j_{2},j_{3}),$  $M=(m_1,m_2,m_2),$with  $m_r, j_r=0,1.$  Let
\begin{gather*}
\chi_{{}_E}^J=\chi_{{}_{E_{1}}}^{j_{1}}\chi_{{}_{E_{2}}}^{j_{2}}\chi_{{}_{E_{3}}}^{j_{3}},  \;\ 
\chi_{{}_F}^M=\chi_{{}_{F_{1}}}^{m_{1}}\chi_{{}_{F_{2}}}^{m_{2}}\chi_{{}_{F_{3}}}^{m_{3}},
\end{gather*}
and write
\begin{gather*}
K_\vkap(\xi,\eta)=  \sum_{j_j+m_j=1}  \chi_{{}_{E}}^J \chi_{{}_{F}}^M K_{\vkap}(\xi,\eta).
\end{gather*}

In the  case where $J=(1,1,1)$ and $M=(0,0,0),$ in virtue of \eqref{ineq-ej1}, we have
\begin{gather*}
\chi_E^J K_\vkap(\xi,\eta)\leq \frac{C}{W_{\vkap}(\xi-\eta)},
\end{gather*}
and hence from \eqref{ineq1},
\begin{gather*}
\int_{\mr^4} \chi_G \chi_E^J K_\vkap(\xi,\eta)^2 d\eta \leq \int_{\mr^4} \frac{C}{W_{\vkap}^2(\xi-\eta)} d\eta=
\int_{\mr^4} \frac{C}{W_{\vkap}^2(\eta)} d\eta <\infty,
\end{gather*}
provided $k_j>\frac{n}{6}$ and $\del$ is small enough.  

Next we consider the case $J=(1,0,1)$ and $M=(0,1,0).$  First, in virtue of the first inequality in \eqref{ineq-ej}, and then because of the definition of $F_{j},$ we have
\begin{gather*}
 K_\vkap(\xi,\eta)=\frac{\lan \xi_1,\xi''\ran^{k_1} \lan \xi_2,\xi''\ran^{k_2}\lan \xi_3,\xi''\ran^{k_3}}{\lan \xi_1-\eta_1,\xi''-\eta''\ran^{k_1} \lan \xi_2-\eta_2,\xi''-\eta''\ran^{k_2}\lan \xi_3-\eta_3,\xi''-\eta''\ran^{k_3}\lan \eta_1,\eta'\ran^{k_1} \lan \eta_2,\eta'\ran^{k_2}\lan \eta_3,\eta''\ran^{k_3}}.
 \end{gather*}
 
 Therefore, using \eqref{ineq-ej} for $j=1,3,$
 \begin{gather*}
  \chi_E^J \chi_F^M K_\vkap(\xi,\eta) \leq C \frac{\lan \xi_2,\xi''\ran^{k_2}}{ W_{\alpha}(\xi-\eta) \lan \eta_2,\eta''\ran^{k_3}} \leq \\
\frac{C }{\lan \xi_1-\eta_1,\xi''-\eta''\ran^{k_1} \lan \xi_3-\eta_3,\xi''-\eta''\ran^{k_3} \lan \eta_2,\eta''\ran^{k_2}}
\end{gather*}
Again, making a change of variables, one finds that for $J=(1,0,1)$ and $M=(0,1,0),$  as in \eqref{ineq1},
\begin{gather*}
\int_{\mr^4} \chi_E^J \chi_F^M K_\alpha(\xi,\eta) d\eta<\infty,
\end{gather*}
provided $\alpha_{jk}>\frac{n}{6}$ and $\del$ is small enough. The other terms are  controlled in the same way, and the details are left to the reader.
\end{proof}

Next we prove Proposition \ref{CIA}.
\begin{proof}  By replacing $f(y,u)$ with $f(y,u)-f(y,0),$ we may assume that $f(y,u)=0$ and since $u$ is compactly supported we may assume that $f(y,u(y))$ is compactly supported in $y.$   

First we need to prove a particular case:
\begin{lemma}\label{GN-case} If $u\in H^{0,k_1,k_2,k_3}(\mrn)\cap L^\infty(\mrn),$ $k_j\in \mn,$ $j=1,2,3,$ and $f(y,s)\in \CI,$ then $f(y,u) \in H^{0,k_1,k_2,k_3}(\mrn).$
\end{lemma}
\begin{proof}  The proof depends on three ingredients:
\begin{enumerate}[1.]
\item A special case of the Galgliardo-Nirenberg inequality, see for example \cite{Fio}: For $|\alpha|\leq m,$
\begin{gather}
|| D_y^\alpha u||_{L^{\frac{2m}{|\alpha|}}} \leq C ||u||_{L^\infty}^{1-\frac{|\alpha|}{m}} \left(\sum_{|\beta|\leq m} ||D^\beta u||_{L^2}\right)^{\frac{|\alpha|}{m}}. \label{GN}
\end{gather}
\item  The following version of H\"older's inequality:
\begin{gather}
|| f_1 f_2 \ldots f_N||_{L^2} \leq ||f_1||_{L^{p_1}}||f_2||_{L^{p_2}}\ldots ||f_N||_{L^{p_N}}, \text{ if }  \sum_{j=1}^N \frac{1}{p_j}=\ha.\label{HOL}
\end{gather}
\item The following formula, which can be proven by induction:
\begin{gather}
\begin{gathered}
(D_{y_1}, D_{y''})^\ga f(y,u)= \sum C_{\beta_1,\ldots \beta_k} (y,u) (D_{y_1},D_{y'})^{\beta_1} u) (D_{y_1},D_{y''})^{\beta_2} u) \ldots 
((D_{y_1},D_{y''})^{\beta_k} u), \\
|\beta_1|+|\beta_2|+\ldots |\beta_k|\leq |\ga|, \;\ k=|\ga|-1. 
\end{gathered}\label{DER}
\end{gather}
\end{enumerate}

Since $u\in L^\infty$ and $f\in \CI,$ it follows from \eqref{DER} and \eqref{HOL} that
\begin{gather*}
||(D_{y_1}, D_{y''})^\ga f(y,u)||_{L^2}\leq C(\ga,||u||_{L^\infty})\sum_{\beta_1,\ldots, \beta_k} ||(D_{y_1},D_{y'})^{\beta_1} u) (D_{y_1},D_{y''})^{\beta_2} u) \ldots 
(D_{y_1},D_{y''})^{\beta_k} u)||_{L^2}\leq \\
C\sum_{\beta_1,\ldots, \beta_k} ||(D_{y_1},D_{y'})^{\beta_1} u)||_{L^{p_1}} ||(D_{y_1},D_{y''})^{\beta_2} u)||_{L^{p_2}} \ldots 
||(D_{y_1},D_{y''})^{\beta_k} u)||_{L^{p_N}}, \\
\text{ where } p_j=\frac{2|\ga|}{|\beta_j|}, \;\ j=1,2, \ldots, N.
\end{gather*}
Now using \eqref{GN} we find that
\begin{gather*}
||(D_{y_1}, D_{y''})^\ga f(y,u)||_{L^2}\leq C (\ga, ||u||_{L^\infty}) ||\left(\sum_{|\beta|\leq m} ||(D_{y_1},D_{y''})^\beta u||_{L^2}\right).
\end{gather*}
We apply the same argument to control $||(D_{y_1}, D_{y''})^{\ga_1}(D_{y_2}, D_{y''})^{\ga_2}(D_{y_3}, D_{y''})^{\ga_3} f(y,u)||_{L^2}.$
This ends the proof of the Lemma.
\end{proof}

Next we prove that Property {\bf P.\ref{CIA}} of Proposition \ref{BSP2} holds for $s=0:$
\begin{proof} Since $\p_{y_j}$ is elliptic, for at least one value of $j\in\{1,2,3\},$ if 
$u\in H_{\loc}^{0-,k_1,k_2,k_3}(\mrn),$ and  $k_j>\frac{n}{6}+1,$ $j=1,2,3,$ one can pick $m_j\in \mn$ such that
$\frac{n}{6} < m_j <k_j,$ and therefore $u\in H^{0,m_1,m_2,m_3}(\mr^n).$  Since $k_j>\frac{n}{6},$ it follows from Property {\bf P.\ref{inc-LI}} that $u\in L^\infty(\mr^n).$ It follows from Lemma \ref{GN-case} that $f(y,u) \in H^{0,m_1,m_2,m_3}(\mr^n)$ with
$1\leq m_j<k_j$ and for all $||u||_{-\del,k_1,k_2,k_3} \leq C,$  there exists $K$ depending on $f,$  $C$ and $\del>0$  such that 
$||f(y,u)||_{0,m_1,m_2,m_3}\leq K.$   But we know that
\begin{gather}
\begin{gathered}
\p_{y_1} f(y,u)= (\p_{y_1} f)(y,u) + (\p_u f)(y,u) \p_{y_1}u, \\
\p_{y_j} f(y,u)= (\p_{y_j} f)(y,u) + (\p_u f)(y,u) \p_{y_j}u, \;\ j\geq 4,
\end{gathered}\label{step1}
\end{gather}

We know that $(\p_{y_1} f)(y,u)\in H^{0,m_1,m_2,m_3}(\mr^n),$ 
$(\p_{u} f)(y,u)\in H^{0,m_1,m_2,m_3}(\mr^n)$ and that $\p_{y_1} u \in H^{0-,k_1-1,k_2,k_3}(\mr^3).$ Since $k_j>m_j>\frac{n}{6},$ and $k_j-1>\frac{n}{6},$ it follows from Proposition \ref{BSP2} that
\begin{gather*}
\p_{y_1} f(y,u) \in H^{0-, r_1,m_2,m_3}(\mr^n), \;\ r_1=\min\{m_1,k_1-1\}, \\
\p_{y_j} f(y,u) \in H^{0-, r_1,m_2,m_3}(\mr^n), \; j\geq 4, \;\ r_1=\min\{m_1,k_1-1\}.
\end{gather*}
This implies that
\begin{gather*}
f(y,u) \in H^{0-, r_1+1,m_2,m_3}(\mr^n), \;\ r_1=\min\{m_1,k_1-1\}.
\end{gather*}
   We have two possibilities: either $r_1+1=k_1$ or
$r_1+1=m_1+1.$  If $r_j+1=m_1+1\in \mn,$ in this case we repeat the argument for $m_1$ replaced by $m_1+1.$ So after finitely many steps, we will find that $r_1+1=k_1$ and so we conclude that
\begin{gather*}
f(y,u)\in H^{0-, k_1,m_2,m_3}(\mr^n).
\end{gather*}
for an arbitrary function  $f\in C^\infty.$   Now we repeat \eqref{step1} for $(\p_{y_2},\p_{y''}),$ and we conclude that
\begin{gather*}
\p_{y_2} f(y,u) \in H^{0-, k_1,r_2,m_3}(\mr^n), \;\ r_2=\min\{m_2,k_2-1\}, \\
\p_{y_j} f(y,u) \in H^{0-, k_1,r_2,m_3}(\mr^n), \; j\geq 4, \;\ r_2=\min\{m_2,k_2-1\}.
\end{gather*}
We apply the same argument and conclude that 
\begin{gather*}
f(y,u)\in H^{0-, k_1,k_2,m_3}(\mr^n).
\end{gather*}
 We do this again with respect to $(\p_{y_3},\p_{y''}),$ and we  are done.
\end{proof}
Now we can prove that Property {\bf P.\ref{CIA}} holds for $s\in \mn.$  We prove this induction.  Suppose $u \in H^{1-,k_1,k_2,k_3}(U).$ Since
\begin{gather*}
\p_{y_j} f(y,u)= (\p_{y_j} f)(y,u)+( \p_u f)(y,u)\p_{y_j} u,
\end{gather*}
and we know the result holds for $s=0,$ so  $(\p_{y_j} f)(y,u), ( \p_u f)(y,u) \in H^{0-,k_1,k_2,k_3}(U),$ Property {\bf P.\ref{prod-clos}}, gives that that $\p_{y_j} f(y,u)\in H^{0-,k_1,k_2,k_3}(U),$ and therefore $f(y,u)\in H_{\loc}^{1-,k_1,k_2,k_3}(U).$ The same argument shows that if $H_{\loc}^{s-1-,k_1,k_2,k_3}(U)$  is a $C^\infty$ algebra, so is $H_{\loc}^{s-,k_1,k_2,k_3}(U).$ The bound on the norm also follows from the proof.
\end{proof}

\section{Acknowledgements}  

The author was a visiting member of the Microlocal Analysis Program of the Mathematical Sciences Research Institute (MSRI) in Berkeley, California, during the fall 2019, when part of this work was done.  He is grateful to the MSRI for their hospitality. His membership at MSRI was supported by the National Science Foundation Grant No. DMS-1440140.

  The author is grateful to the Simons Foundation  for its support under grant \#349507, Ant\^onio S\'a Barreto.

\end{document}